\documentclass[leqno,final]{siamltex}
\usepackage{t1enc}
\usepackage{amssymb}
\usepackage{amsmath}
\usepackage{latexsym}
\usepackage{stmaryrd}
\usepackage{wasysym}
\usepackage{booktabs}
\usepackage{multicol}
\usepackage{graphicx}
\usepackage{paralist}
\usepackage[algoruled,lined,algosection,longend]{algorithm2e}
\usepackage{arydshln}
\usepackage[mathcal]{euscript}

\newcommand{\assgn}{\ensuremath\mathrel{\mathop:}=}
\newcommand{\sgn}{\operatorname{sgn}}

\newcommand{\Sym}{\operatorname{\mathsf{Sym}}}
\newcommand{\FMA}{\operatorname{\mathtt{fma}}}
\newcommand{\RCP}{\operatorname{\mathtt{rcp}}}
\newcommand{\RSQRT}{\operatorname{\mathtt{rsqrt}}}

\newcommand{\bdiv}{\operatorname{\mathbf{div}}}
\newcommand{\ct}{\operatorname{ct}}
\newcommand{\tn}{\operatorname{tn}}
\newcommand{\cs}{\operatorname{cs}}

\newcommand{\rn}{\operatorname{rn}}
\newcommand{\ru}{\operatorname{ru}}
\newcommand{\rz}{\operatorname{rz}}
\newcommand{\fl}{\operatorname{\mathit{f\mkern-2mu\ell}}}
\newcommand{\hpz}{\hphantom{0}}
\title{A hierarchically blocked Jacobi SVD algorithm
  for~single~and~multiple~graphics~processing~units\thanks{This work
    was supported in part by grant 037--1193086--2771 from Ministry of
    Science, Education and Sports, Republic of Croatia, and by
    NVIDIA's Academic Partnership Program.}}
\author{Vedran Novakovi\'{c}\footnotemark[2]}
\begin{document}
\maketitle
\renewcommand{\thefootnote}{\fnsymbol{footnote}}
\footnotetext[2]{University of Zagreb, Faculty of Mechanical
  Engineering and Naval Architecture, I.~Lu\v{c}i\'{c}a 5, HR-10000
  Zagreb, Croatia (venovako@fsb.hr),
  http://www.fsb.unizg.hr/venovako/.}
\renewcommand{\thefootnote}{\arabic{footnote}}
\begin{abstract}
  We present a hierarchically blocked one-sided Jacobi algorithm for
  the singular value decomposition (SVD), targeting both single and
  multiple graphics processing units (GPUs).  The blocking structure
  reflects the levels of GPU's memory hierarchy.  The algorithm may
  outperform MAGMA's \texttt{dgesvd}, while retaining high relative
  accuracy.  To this end, we developed a family of parallel pivot
  strategies on GPU's shared address space, but applicable also to
  inter-GPU communication.  Unlike common hybrid approaches, our
  algorithm in a single GPU setting needs a CPU for the controlling
  purposes only, while utilizing GPU's resources to the fullest extent
  permitted by the hardware.  When required by the problem size, the
  algorithm, in principle, scales to an arbitrary number of GPU nodes.
  The scalability is demonstrated by more than twofold speedup for
  sufficiently large matrices on a Tesla S2050 system with four GPUs
  vs.~a single Fermi card.
\end{abstract}
\pagestyle{myheadings}
\thispagestyle{plain}
%
%
\section{Introduction}\label{sec:1}
%
%
Graphics processing units have become a widely accepted tool of
parallel scientific computing, but many of the established algorithms
still need to be redesigned with massive parallelism in mind.  Instead
of multiple CPU cores, which are fully capable of simultaneously
processing different operations, GPUs are essentially limited to many
concurrent instructions of the same kind---a paradigm known as SIMT
(single-instruction, multiple-threads) parallelism.

SIMT type of parallelism is not the only reason for the redesign.
Modern CPU algorithms rely on (mostly automatic) multi-level cache
management for speedup.  GPUs instead offer a complex memory
hierarchy, with different access speeds and patterns, and both
automatically and programmatically managed caches.  Even more so than
in the CPU world, a (less) careful hardware-adapted blocking of a GPU
algorithm is the key technique by which considerable speedups are
gained (or lost).

After the introductory paper~\cite{Novakovic-SingerSanja-2011}, here
we present a family of the full
block~\cite{Hari-SingerSanja-SingerSasa-2014} and the
block-oriented~\cite{Hari-SingerSanja-SingerSasa-2010} one-sided
Jacobi-type algorithm variants for the ordinary (SVD) and the
hyperbolic singular value decomposition (HSVD) of a matrix, targeting
both a single and the multiple GPUs.  The blocking of our algorithm
follows the levels of the GPU memory hierarchy; namely, the innermost
level of blocking tries to maximize the amount of computation done
inside the fastest (and smallest) memory of the registers and manual
caches.  The GPU's global RAM and caches are considered by the
mid-level, while inter-GPU communication and synchronization are among
the issues addressed by the outermost level of blocking.

At each blocking level an instance of either the block-oriented or the
full block Jacobi (H)SVD is run, orthogonalizing pivot columns or
block-columns by conceptually the same algorithm at the lower level.
Thus, the overall structure of the algorithm is hierarchical (or
recursive) in nature, and ready to fit not only the current GPUs, but
also various other memory and communication hierarchies, provided that
efficient, hardware-tuned implementations at each level are
available.

The Jacobi method is an easy and elegant way to find the eigenvalues
and eigenvectors of a symmetric matrix.  In 1958
Hestenes~\cite{Hestenes-58} developed the one-sided Jacobi SVD
method---an implicit diagonalization is performed by orthogonalizing a
factor of a symmetric positive definite matrix.  But, after discovery
of the QR algorithm in 1961/62 by Francis and Kublanovskaya, the
Jacobi algorithm seemed to have no future, at least in the sequential
processing world, due to its perceived slowness~\cite{Eberlein-87}.
However, a new hope for the algorithm has been found in its
amenability to parallelization, in its proven high relative
accuracy~\cite{Demmel-Veselic-92}, and finally in the emergence of the
fast Jacobi SVD implementation in LAPACK, due to Drma\v{c} and
Veseli\'{c}~\cite{Drmac-Veselic-2008,Drmac-Veselic-2008a}.

In the beginning of the 1970s Sameh in~\cite{Sameh-71} developed two
strategies for parallel execution of the Jacobi method on Illiac IV\@.
The first of those, the modulus strategy, is still in use, and it is
one of the very rare parallel strategies for which a proof of
convergence exists~\cite{LukF-Park-89a}.

In the mid 1980s, Brent and Luk designed another parallel
strategy~\cite{Brent-LukF-85}, known by the names of its creators.
The same authors, together with Van Loan~\cite{Brent-LukF-VanLoan-85},
described several parallel one-sided Jacobi and Kogbetliantz (also
known as ``the two-sided Jacobi'') algorithms.  The parallel block
Kogbetliantz method is developed in~\cite{VanLoan-86}.

In 1987 Eberlein~\cite{Eberlein-87} proposed two strategies, the
round-robin strategy, and another one that depends on the parity of a
sweep.  A new efficient recursive divide-exchange parallel strategy,
specially designed for the hypercube topologies (and, consequently,
matrices of order $2^n$) is given in~\cite{Gao-Thomas-88}.  This
strategy is later refined by Mantharam and Eberlein
in~\cite{Mantharam-Eberlein-93} to the block-recursive (BR) strategy.

Two papers by Luk and Park~\cite{LukF-Park-89,LukF-Park-89a} published
in 1989 established equivalence between numerous strategies, showing
that if one of them is convergent, then all equivalent strategies are
convergent.  In the same year Shroff and
Schreiber~\cite{Shroff-Schreiber-89} showed convergence for a family
of strategies called the wavefront ordering, and discussed the
parallel orderings weakly equivalent to the wavefront ordering, and
thus convergent.

One of the first attempts of a parallel SVD on a GPU is a hybrid one,
by Lahabar and Narayanan~\cite{Lahabar-Narayanan-2009}.  It is based
on the Golub--Reinsch algorithm, with bidiagonalization and updating
of the singular vectors performed on a GPU, while the rest of the
bidiagonal QR algorithm is computed on a CPU\@.  In
MAGMA\footnote{Matrix Algebra on GPU and Multicore Architectures,
  http://icl.utk.edu/magma/}, a GPU library of the LAPACK-style
routines, \texttt{dgesvd} algorithm is also hybrid, with
bidiagonalization (\texttt{DGEBRD}) parallelized on a
GPU~\cite{Tomov-Nath-Dongarra-2010}, while for the bidiagonal QR,
LAPACK routine \texttt{DBDSQR} is used.  We are unaware of any
multi-GPU SVD implementations.

In two of our previous
papers~\cite{SingerSanja-SingerSasa-Novakovic-Uscumlic-Dunjko-2012,SingerSanja-SingerSasa-Novakovic-Davidovic-Bokulic-Uscumlic-2012}
we discussed the parallel one-sided Jacobi algorithms for the
hyperbolic SVD with two and three levels of blocking, respectively.
The outermost level is mapped to a ring of CPUs which communicate
according to a slightly modified modulus strategy, while the inner two
(in the three-level case) are sequential and correspond to the
``fast'' (L1) and ``slow'' (L2 and higher) cache levels.

At first glance a choice of the parallel strategy might seem as a
technical detail, but our tests at the outermost level have shown that
the modified modulus strategy can be two times faster than the
round-robin strategy.  That motivated us to explore if and how even
faster strategies could be constructed, that preserve the accuracy of
the algorithm.  We present here a class of parallel strategies
designed around a conceptually simple but computationally difficult
notion of a metric on a set of strategies of the same order.  These
new strategies can be regarded as generalizations of the
Mantharam--Eberlein BR strategy to all even matrix orders,
outperforming the Brent and Luk and modified modulus strategies in our
GPU algorithm.

However, a parallel strategy alone is not sufficient to achieve decent
GPU performance.  The standard routines that constitute a block Jacobi
algorithm, like the Gram matrix formation, the Cholesky (or the QR)
factorization, and the pointwise one-sided Jacobi algorithm itself,
have to be mapped to the fast, but in many ways limited shared memory
of a GPU, and to the peculiar way the computational threads are
grouped and synchronized.  Even the primitives that are usually taken
for granted, like the numerically robust calculation of a vector's
$2$-norm, present a challenge on a SIMT architecture.  Combined with
the problems inherent in the block Jacobi algorithms, whether
sequential or parallel, like the reliable convergence criterion, a
successful design of the Jacobi-type GPU (H)SVD is far from trivial.

In this paper we show that such GPU-centric design is possible and
that the Jacobi-type algorithms for a single and the multiple GPUs
compare favorably to the present state-of-the-art in the GPU-assisted
computation of the (H)SVD\@.  Since all computational work is
offloaded to a GPU, we need no significant CPU $\leftrightarrow$ GPU
communication nor complex synchronization of their tasks.  This
facilitates scaling to a large number of GPUs, while keeping their
load in balance and communication simple and predictable.  While many
questions remain open, we believe that the algorithms presented here
are a valuable choice to consider when computing the (H)SVD on the
GPUs.

The paper is organized as follows.  In Section~\ref{sec:2} a brief
summary of the one-sided Jacobi-type (H)SVD block algorithm variants
is given.  In Section~\ref{sec:3} new parallel Jacobi
strategies---nearest to row-cyclic and to column-cyclic are
developed.  The main part of the paper is Section~\ref{sec:4}, where a
detailed implementation of a single-GPU Jacobi (H)SVD algorithm is
described.  In Section~\ref{sec:5}, a proof-of-concept implementation
on multiple GPUs is presented.  In Section~\ref{sec:6}, results of the
numerical testing are commented.  Two appendices complete the paper
with a parallel, numerically stable procedure for computing the
$2$-norm of a vector, and some considerations about the Jacobi
rotation formulas.
%
%
\section{Jacobi--type SVD algorithm}\label{sec:2}
%
%
Suppose that a matrix $G\in\mathbb{F}^{m\times n}$, where $\mathbb{F}$
denotes the real ($\mathbb{R}$) or the complex ($\mathbb{C}$) field,
is given.  Without loss of generality, we may assume that $m \geq n$.
If not, instead of $G$, the algorithm will transform $G^{\ast}$.

If $m\gg n$, or if the column rank of $G$ is less than $n$, then the
first step of the SVD is to preprocess $G$ by the QR factorization
with column pivoting~\cite{Drmac-97a} and, possibly, row pivoting or
row presorting,
\begin{equation}
  G = P_r Q R P_c
  = P_r Q \begin{bmatrix}
    R_0 \\
    0
  \end{bmatrix} P_c,
\label{2.1}
\end{equation}
where $Q$ is unitary, $R_0\in\mathbb{F}^{k\times n}$ is upper
trapezoidal with the full row rank $k$, while $P_r$ and $P_c$ are
permutations.  If $k<n$, then $R_0$ should be factored by the LQ
factorization,
\begin{equation}
  R_0 = P'_r L Q'_{} P'_c
  = P'_r \begin{bmatrix}
    L_0 & 0
  \end{bmatrix} Q'_{} P'_c.
\label{2.2}
\end{equation}
Finally, $L_0\in\mathbb{F}^{k\times k}$ is a lower triangular matrix
of full rank.  From the SVD of $L_0$, by (\ref{2.1}) and (\ref{2.2}),
it is easy to compute the SVD of $G$.  Thus, we can assume that the
initial $G$ is square and of full rank $n$, with $n \geq 2$.

The one-sided Jacobi SVD algorithm for $G$ can be viewed as the
implicit two-sided Jacobi algorithm which diagonalizes either
$G^{\ast}G$ or $GG^{\ast}$.  Let, e.g., $H\assgn G^{\ast}G$.
Stepwise, a suitably chosen pair of pivot columns $g_p$ and $g_q$ of
$G$ is orthogonalized by postmultiplying the matrix
$\begin{bmatrix}g_p^{} & g_q^{}\end{bmatrix}$ by a Jacobi plane
rotation $\widehat{V}_{pq}$, which diagonalizes the $2\times 2$ pivot
matrix $\widehat{H}_{pq}$,
\begin{equation}
  \widehat{H}_{pq} = \begin{bmatrix}
    h_{pp}^{} & h_{pq}^{}\\
    h_{pq}^{\ast} & h_{qq}^{}
  \end{bmatrix}
  = \begin{bmatrix}
    g_p^{\ast} g_p^{} & g_p^{\ast} g_q^{}\\
    g_q^{\ast} g_p^{} & g_q^{\ast} g_q^{}
  \end{bmatrix}
  = \begin{bmatrix}
    g_p^{\ast}\\
    g_q^{\ast}
  \end{bmatrix}
  \begin{bmatrix}
    g_p^{}&
    g_q^{}
  \end{bmatrix},
\label{2.3}
\end{equation}
such that
\begin{equation}
  \widehat{V}_{pq}^{\ast} \widehat{H}_{pq}^{} \widehat{V}_{pq}^{} = \diag(\hat{\lambda}_p^{}, \hat{\lambda}_q^{}).
\label{2.4}
\end{equation}

In case of convergence, the product of transformation matrices will
approach the set of eigenvector matrices.  Let $V$ be an eigenvector
matrix of $H$. Then
\begin{displaymath}
  \Lambda = V^{\ast}HV = (V^{\ast}G^{\ast})(GV), \quad
  \Lambda = \diag(\lambda_1, \lambda_2, \ldots, \lambda_n).
\end{displaymath}
The resulting matrix $GV$ has orthogonal columns, and can be written
as
\begin{equation}
  GV=U\Sigma,
\label{2.5}
\end{equation}
where $U$ is unitary and $\Sigma = \Lambda^{1/2}$ is a diagonal matrix
of the column norms of $GV$.

The matrix $U$ of the left singular vectors results from scaling the
columns of $GV$ by $\Lambda^{-1/2}$, so only the right singular
vectors $V$ have to be obtained, either by accumulation of the Jacobi
rotations applied to $G$, or by solving the linear system~(\ref{2.5})
for $V$, with the initial $G$ preserved.  The system (\ref{2.5}) is
usually triangular, since $G$ is either preprocessed in such a form,
or already given as a Cholesky factor in an eigenproblem computation.
Solving (\ref{2.5}) is therefore faster than accumulation of $V$, but
it needs more memory and may be less accurate if $G$ is not
well-conditioned (see~\cite{Drmac-99}).

The choice of pivot indices $p$, $q$ in successive steps is essential
for possible parallelization of the algorithm.  We say that two pairs
of indices, $(p,q)$ and $(p',q')$, are \emph{disjoint\/}, or
\emph{non-colliding\/}, if $p \neq q$, $p' \neq q'$, and
$\{p,q\} \cap \{p',q'\} = \emptyset$.  Otherwise, the pairs are called
\emph{colliding\/}.  These definitions are naturally extended to an
arbitrary number of pairs.  The pairs of indexed objects (e.g., the
pairs of matrix columns) are disjoint or (non-)colliding, if such are
the corresponding pairs of the objects' indices.

The one-sided Jacobi approach is better suited for parallelization
than the two-sided one, since it can simultaneously process disjoint
pairs of columns.  This is still not enough to make a respectful
parallel algorithm.  In the presence of a memory hierarchy, the
columns of $G$ and $V$ should be grouped together into block-columns,
\begin{equation}
  G = \begin{bmatrix}
    G_{\mathsf{1}} & G_{\mathsf{2}} & \cdots & G_{\mathsf{b}}
  \end{bmatrix}, \quad
  V = \begin{bmatrix}
    V_{\mathsf{1}} & V_{\mathsf{2}} & \cdots & V_{\mathsf{b}}
  \end{bmatrix}.
\label{2.6}
\end{equation}
In order to balance the workload, the block-columns should be (almost)
equally sized.

Usually, a parallel task processes two block-columns $G_{\mathsf{p}}$
and $G_{\mathsf{q}}$, i.e., a single pivot block-pair, either by
forming the pivot block-matrix $H_{\mathsf{p}\mathsf{q}}$ and its
Cholesky factor $R_{\mathsf{p}\mathsf{q}}$,
\begin{equation}
  H_{\mathsf{p}\mathsf{q}} = \begin{bmatrix}
    G_{\mathsf{p}}^{\ast} G_{\mathsf{p}}^{} & G_{\mathsf{p}}^{\ast} G_{\mathsf{q}}^{}\\
    G_{\mathsf{q}}^{\ast} G_{\mathsf{p}}^{} & G_{\mathsf{q}}^{\ast} G_{\mathsf{q}}^{}
  \end{bmatrix}
  = \begin{bmatrix}
    G_{\mathsf{p}}^{\ast}\\
    G_{\mathsf{q}}^{\ast}
  \end{bmatrix}
  \begin{bmatrix}
    G_{\mathsf{p}}^{} & G_{\mathsf{q}}^{}
  \end{bmatrix}, \quad
  P^{\ast} H_{\mathsf{p}\mathsf{q}}^{} P^{} = R_{\mathsf{p}\mathsf{q}}^{\ast} R_{\mathsf{p}\mathsf{q}}^{},
\label{2.7}
\end{equation}
or by shortening the block-columns
$\begin{bmatrix}G_{\mathsf{p}} & G_{\mathsf{q}}\end{bmatrix}$
directly, by the QR factorization,
\begin{equation}
  \begin{bmatrix}
    G_{\mathsf{p}} & G_{\mathsf{q}}
  \end{bmatrix} P
  = Q_{\mathsf{p}\mathsf{q}}
  \begin{bmatrix}
    R_{\mathsf{p}\mathsf{q}} \\ 0
  \end{bmatrix}.
\label{2.8}
\end{equation}

The diagonal pivoting in the Cholesky factorization, or analogously,
the column pivoting in the QR factorization should be employed, if
possible
(see~\cite{SingerSanja-SingerSasa-Novakovic-Uscumlic-Dunjko-2012} for
further discussion, involving also the hyperbolic SVD case).  However,
the pivoting in factorizations (\ref{2.7}) or (\ref{2.8}) may be
detrimental to performance of the parallel implementations of the
respective factorizations, so their non-pivoted counterparts have to
be used in those cases (with $P = I$).  Either way, a square pivot
factor $R_{\mathsf{p}\mathsf{q}}$ is obtained.  Note that the unitary
matrix $Q_{\mathsf{p}\mathsf{q}}$ in the QR factorization is not
needed for the rest of the Jacobi process, and it consequently does
not have to be computed.

Further processing of $R_{\mathsf{p}\mathsf{q}}$ is determined by a
variant of the Jacobi algorithm.  The following variants are
advisable: \emph{block-oriented\/} variant
(see~\cite{Hari-SingerSanja-SingerSasa-2010}), when the communication
(or memory access) overhead between the tasks is negligible compared
to the computational costs, and \emph{full block\/} variant
(see~\cite{Hari-SingerSanja-SingerSasa-2014}), otherwise.

In both variants, $R_{\mathsf{p}\mathsf{q}}$ is processed by an inner
one-sided Jacobi method.  In the block-oriented variant, exactly one
(quasi-)sweep of the inner (quasi-)cyclic\footnote{See Section
  \ref{sec:3} for the relevant definitions.} Jacobi method is allowed.
Therefore, $R_{\mathsf{p}\mathsf{q}}$ is transformed to
$R_{\mathsf{p}\mathsf{q}}' = R_{\mathsf{p}\mathsf{q}}^{} \widetilde{V}_{\mathsf{p}\mathsf{q}}^{}$,
with $\widetilde{V}_{\mathsf{p}\mathsf{q}}$ being a product of the
rotations applied in the (quasi-)sweep.  In the full block variant,
the inner Jacobi method computes the SVD of
$R_{\mathsf{p}\mathsf{q}}$, i.e.,
$R_{\mathsf{p}\mathsf{q}} V_{\mathsf{p}\mathsf{q}} = U_{\mathsf{p}\mathsf{q}} \Sigma_{\mathsf{p}\mathsf{q}}$.
By $V_{\mathsf{p}\mathsf{q}}'$ we denote the transformation matrix,
either $\widetilde{V}_{\mathsf{p}\mathsf{q}}^{}$ from the former, or
$V_{\mathsf{p}\mathsf{q}}$ from the latter variant.

Especially for the full block variant, the width of the block-columns
should be chosen such that $R_{\mathsf{p}\mathsf{q}}$ and
$V_{\mathsf{p}\mathsf{q}}'$ jointly saturate, without being evicted
from, the fast local memory (e.g., the private caches) of a processing
unit to which the block-columns
$\begin{bmatrix}G_{\mathsf{p}} & G_{\mathsf{q}}\end{bmatrix}$ are
assigned.  This also allows efficient blocking of the matrix
computations in (\ref{2.7}) (or (\ref{2.8})) and (\ref{2.9}), as
illustrated in Subsections \ref{sec:4.1} and \ref{sec:4.4}.

Having computed $V_{\mathsf{p}\mathsf{q}}'$, the block-columns of $G$
(and, optionally, $V$) are updated,
\begin{equation}
  \begin{bmatrix}
    G_{\mathsf{p}}' & G_{\mathsf{q}}'
  \end{bmatrix}
  = \begin{bmatrix}
    G_{\mathsf{p}}^{} & G_{\mathsf{q}}^{}
  \end{bmatrix}
  V_{\mathsf{p}\mathsf{q}}',\quad
  \begin{bmatrix}
    V_{\mathsf{p}}' & V_{\mathsf{q}}'
  \end{bmatrix}
  = \begin{bmatrix}
    V_{\mathsf{p}}^{} & V_{\mathsf{q}}^{}
  \end{bmatrix}
  V_{\mathsf{p}\mathsf{q}}'.
\label{2.9}
\end{equation}
The tasks processing disjoint pairs of block-columns may compute
concurrently with respect to each other, up to the local completions
of updates (\ref{2.9}).  A task then replaces (at least) one of its
updated block-columns of $G$ by (at least) one updated block-column of
$G$ from another task(s).  Optionally, the same replacement pattern is
repeated for the corresponding updated block-column(s) of $V$.  The
block-column replacements entail a synchronization of the tasks.  The
replacements are performed by communication or, on shared-memory
systems, by assigning a new pivot block-pair to each of the tasks.

The inner Jacobi method of both variants may itself be blocked, i.e.,
may divide $R_{\mathsf{p}\mathsf{q}}$ into block-columns of an
appropriate width for the next (usually faster but smaller) memory
hierarchy level.  This recursive blocking principle terminates at the
pointwise (non-blocked) Jacobi method, when no advantages in
performance could be gained by further blocking.  In that way a
hierarchical (or multi-level) blocking algorithm is created, with each
blocking level corresponding to a distinct communication or memory
domain
(see~\cite{SingerSanja-SingerSasa-Novakovic-Davidovic-Bokulic-Uscumlic-2012}).

For example, in the case of a multi-GPU system, we identify access to
the global memory (RAM) of a GPU as slow compared to the shared memory
and register access, and data exchange with another GPU as slow
compared to access to the local RAM\@.  This suggests the two-level
blocking for a single-GPU algorithm, and the three-level for a
multi-GPU one.

The inner Jacobi method, whether blocked or not, may be sequential or
parallel.  Both a single-GPU and a multi-GPU algorithm are examples of
a nested parallelism.

Similar ideas hold also for the hyperbolic SVD (HSVD)\@.  If
$G \in \mathbb{F}^{m \times n}$, $m \geq n$, and
$\rank(G) = \rank(G^{} J^{} G^{\ast})$, where $J = \diag(\pm 1)$, then
the HSVD of $G$ is (see~\cite{Onn-Steinhardt-Bojanczyk-1991,Zha-96})
\begin{equation}
  G = U \begin{bmatrix}
    \Sigma \\
    0
  \end{bmatrix}
  V^{\ast}, \quad \Sigma = \diag(\sigma_1, \ldots, \sigma_n), \quad
  \sigma_1 \geq \sigma_2 \geq \cdots \sigma_n \geq 0.
\label{2.10}
\end{equation}
Here, $U$ is a unitary matrix of order $m$, while $V$ is $J$-unitary,
(i.e., $V^{\ast} J V = J$) of order $n$.  The HSVD in (\ref{2.10}) can
be computed by orthogonalization of, either the of columns $G^{\ast}$
by trigonometric rotations~\cite{Dopico-Koev-Molera-2009}, or the
columns of $G$ by hyperbolic rotations~\cite{Veselic-93}.

A diagonalization method for the symmetric definite (or indefinite)
matrices requires only the partial SVD (or HSVD), i.e., the matrix $V$
is not needed.  With the former algorithm, the eigenvector matrix $U$
should be accumulated, but with the latter, it is easily obtainable by
scaling the columns of the final $G$.  Thus, the hyperbolic algorithm
is advantageous for the eigenproblem applications, as shown
in~\cite{SingerSanja-SingerSasa-Novakovic-Uscumlic-Dunjko-2012}.

In the sequel we assume that $\mathbb{F}=\mathbb{R}$, but everything,
save the computation of the Jacobi rotations and the hardware-imposed
block sizes, remains also valid for $\mathbb{F}=\mathbb{C}$.
%
%
\section{Parallel pivot strategies}\label{sec:3}
%
%
In each step of the classical, two-sided Jacobi (eigenvalue)
algorithm, the pivot strategy seeks and annihilates an off-diagonal
element $h_{pq}$ with the largest magnitude.  This approach has been
generalized for the parallel two-sided block-Jacobi
methods~\cite{Becka-Oksa-Vajtersic-2002}.  However, the one-sided
Jacobi algorithms would suffer from a prohibitive overhead of forming
and searching through the elements of $H$.  In the parallel algorithm
there is an additional problem of finding $\lfloor n/2 \rfloor$
off-diagonal elements with large magnitudes, that can be
simultaneously annihilated.  Therefore, a cyclic pivot strategy---a
repetitive, fixed order of annihilation of all off-diagonal elements
of $H$---is more appropriate for the one-sided algorithms.

\looseness=-1
More precisely, let $\mathsf{P}_n$ be the set
$\{(i,j) \mid 1 \leq i < j \leq n\}$ of all pivot pairs, i.e., pairs
of indices of the elements in the strictly upper triangle of a matrix
of order $n$, and let $\tau = |\mathsf{P}_n|$ be the cardinality of
$\mathsf{P}_n$.  Obviously, $\tau = n(n-1)/2$.
A \emph{pivot strategy\/} is a function
$\mathcal{P}_n\colon\mathbb{N} \to \mathsf{P}_n$, that associates with
each step $k \geq 1$ a pivot pair $(p(k), q(k))$.

If $\mathcal{P}_n$ is a periodic function, with the fundamental
period $\upsilon$, then, for all $i \geq 1$, the pivot sequences
$\mathsf{C}_i(\upsilon) = (\mathcal{P}_n(k) \mid (i-1)\upsilon+1 \leq k \leq i\upsilon)$,
of length $\upsilon$, are identical.  Consider a case where such a
sequence contains all the pivot pairs from $\mathsf{P}_n$.  Then, if
$\upsilon = \tau$, $\mathcal{P}_n$ is called a \emph{cyclic\/}
strategy and $\mathsf{C}_i(\upsilon)$ is its $i$-th \emph{sweep\/}.
Otherwise, if $\upsilon \geq \tau$, $\mathcal{P}_n$ is called a
\emph{quasi-cyclic\/} strategy and $\mathsf{C}_i(\upsilon)$ is its
$i$-th \emph{quasi-sweep\/}.

A Jacobi method is called (quasi-)cyclic if its pivot strategy is
(quasi-)cyclic.  In the (quasi-)cyclic method the pivot pair therefore
runs through all elements of $\mathsf{P}_n$ exactly (at least) once in
a (quasi-)sweep, and repeats the same sequence until the convergence
criteria are met.

We refer the reader to the standard terminology of equivalent,
shift-equivalent and weakly equivalent
strategies~\cite{Shroff-Schreiber-89}.  In the sequel, we identify a
(quasi-)cyclic pivot strategy with its first (quasi-)sweep, to
facilitate applications of the existing results for finite sequences
to the infinite but periodic ones.

A cyclic Jacobi strategy is \emph{perfectly parallel\/} (p-strategy)
if it allows simultaneous annihilation of as many elements of $H$ as
possible. More precisely, let
\begin{equation}
  t = \left\lfloor \frac{n}{2} \right\rfloor, \quad
  s =
  \begin{cases}
    n-1, & \text{$n$ even,}\\
    n,   & \text{$n$ odd,}
  \end{cases}
\label{3.1}
\end{equation}
then exactly $t$ disjoint pivot pairs can be simultaneously processed
in each of the $s$ parallel steps (p-steps).  As the p-strategies for
an even $n$ admit more parallelism within a p-step, i.e., one parallel
task more than the p-strategies for $n-1$, with the same number of
p-steps in both cases, in the sequel we assume $n$ to be even.

We now provide a definition of a p-strategy \emph{closest\/} to a
given sequential strategy.  The motivation was to explore whether a
heuristic based on such a notion could prove valuable in producing
fast p-strategies from the well-known row- and column-cyclic
sequential strategies.  The numerical testing (see
Section~\ref{sec:6}) strongly supports an affirmative answer.

\looseness=-1
Let $\mathcal{O}$ defines a cyclic pivot strategy.  Then, for each
pivot pair $(i, j) \in \mathsf{P}_n$ there exists an integer $k$ such
that $(i, j) = (p(k), q(k))$, where $(p(k), q(k)) \in \mathcal{O}$.
For any cyclic strategy $\mathcal{O}'$, and for each
$(p'(k), q'(k)) \in \mathcal{O}'$, there is
$(p(\ell(k)), q(\ell(k))) \in \mathcal{O}$, such that
\begin{equation}
  (p'(k), q'(k)) = (p(\ell(k)), q(\ell(k))).
\label{3.2}
\end{equation}
For $1 \leq k \leq \tau$, the values $\ell(k)$ are all distinct, and
lie between $1$ and $\tau$, inclusive.  For a fixed strategy
$\mathcal{O}$, this induces a one-to-one mapping $I_{\mathcal{O}}$,
from the set of all cyclic strategies on matrices of order $n$ to the
symmetric group $\Sym(\tau)$, as
\begin{displaymath}
  I_{\mathcal{O}}(\mathcal{O}') = (\ell(1), \ell(2), \ldots, \ell(k), \ldots, \ell(\tau)) \in \Sym(\tau),
\end{displaymath}
with $\ell(k)$ defined as in (\ref{3.2}).

\begin{definition}
  For any two cyclic strategies, $\mathcal{O}_1$ and $\mathcal{O}_2$,
  we say that $\mathcal{O}_1$ is \emph{closer\/} to $\mathcal{O}$ than
  $\mathcal{O}_2$, and denote that by
  $\mathcal{O}_1\preceq_{\mathcal{O}}\mathcal{O}_2$, if
  $I_{\mathcal{O}}(\mathcal{O}_1) \preceq I_{\mathcal{O}}(\mathcal{O}_2)$,
  where $\preceq$ stands for the lexicographic ordering of
  permutations.
\label{def:3.1}
\end{definition}

The relation ``strictly closer to $\mathcal{O}$'', denoted by
$\prec_{\mathcal{O}}$, is defined similarly.  Note that
$\preceq_{\mathcal{O}}$ is a total order on the finite set of all
cyclic strategies with a fixed $n$, and therefore, each non-empty
subset (e.g., a subset of all p-strategies) has a least element.  Now,
take $\mathcal{O} \in \{\mathcal{R}_n, \mathcal{C}_n\}$, where
$\mathcal{R}_n$ and $\mathcal{C}_n$ are the row-cyclic and the
column-cyclic strategies, respectively.  Then there exists a unique
p-strategy $\mathcal{R}_n^{\parallel}$
(resp.\ $\mathcal{C}_n^{\parallel}$) that is closest to
$\mathcal{R}_n$ (resp.\ $\mathcal{C}_n$).

Interpreted in the graph-theoretical setting, a task of finding the
closest p-strategy amounts to a recursive application of an algorithm
for generating all maximal independent sets (MIS) in lexicographic
order (e.g.,~\cite{JohnsonD-Yannakakis-Papadimitriou-88}).  Let
$\mathsf{G}$ be a simple graph with the vertices enumerated from $1$
to $\tau$, representing pivot pairs from a prescribed cyclic strategy
$\mathcal{O}_n$, and the edges denoting that two pivot pairs collide
(share an index).  Note that $|\text{MIS}(\mathsf{G})| \leq t$, where
$t$ is defined by~(\ref{3.1}).  Then a $\text{MIS}(\mathsf{G})$ with
$t$ vertices is an admissible p-step, and vice versa.  The same holds
for the graph $\mathsf{G}' = \mathsf{G} \setminus S$, where $S$ is any
admissible p-step.

Since any permutation of pivot pairs in a p-step generates an
equivalent (called step-equivalent) p-strategy, the vertices in each
MIS can be assumed to be sorted in ascending order.  With a routine
\texttt{next\_lex}, returning the lexicographically next
$\mathrm{MIS}$ with $t$ vertices (or $\emptyset$ if no such sets are
left), Alg.~\ref{alg:3.1} always produces
$\mathcal{O}_n^{\parallel}$, the p-strategy closest to
$\mathcal{O}_n$.  Note that, at the suitable recursion depths,
\texttt{next\_lex} could prepare further candidates in parallel with
the rest of the search, and parallel searches could also be launched
(or possibly canceled) on the waiting candidates.

\begin{algorithm}{\small
\SetKwInOut{KwDesc}{Description}
\SetKwRepeat{Loop}{begin loop}{end loop}
\SetKwFunction{genstrat}{gen\_strat}
\SetKwFunction{nextlex}{next\_lex}
\SetKwFunction{Append}{append}
\SetKwFunction{Remove}{remove}
\SetKw{InArg}{in}
\SetKw{Boolean}{boolean}
\SetKw{Return}{return}
\KwDesc{Input: the graph $\mathsf{G}$ induced by $\mathcal{O}_n$.
  Output: $\mathcal{O}_n^{\parallel}$ (initially $\emptyset$).}
\BlankLine
\Boolean\genstrat{$\InArg\ \mathsf{G}$}\;
\Begin{
    \lIf(\tcp*[f]{no more pivot pairs (success)})
        {$\mathsf{G}=\emptyset$}{\Return\textit{true\/}}
    \Loop
        {}
        {$S\leftarrow\nextlex{$\mathsf{G}$}$\tcp*{take a lexicographically next MIS\ldots}
        \lIf(\tcp*[f]{\ldots but there are none; fail})
        {$S=\emptyset$}{\Return\textit{false\/}}
        \Append $S$ to $\mathcal{O}_n^{\parallel}$\tcp*{\ldots else, $S$ is a new p-step candidate}
        \lIf(\tcp*[f]{try recursively\ldots})
        {$\genstrat{$\mathsf{G}\setminus S$}$}{\Return\textit{true\/}}
        \Remove $S$ from the back of $\mathcal{O}_n^{\parallel}$\tcp*{\ldots and backtrack if failed}
        }
  }}%
\caption{MIS-based generation of the p-strategy
  $\mathcal{O}_n^{\parallel}$ closest to $\mathcal{O}_n$.}
\label{alg:3.1}
\end{algorithm}

Alg.~\ref{alg:3.1}, however optimized, might still not be
feasible even for the off-line strategy generation, with $n$
sufficiently large.  However, there are two remedies: first, no large
sizes are needed due to the multi-level blocking; and second, we show
in the sequel that it might suffice to generate
$\mathcal{R}_n^{\parallel}$ (or $\mathcal{C}_n^{\parallel}$) only for
$n=2o$, with $o$ odd.

\begin{lemma}
  For all $n$, the sequence of pivot pairs
  $S_n^{(1)} = ((2k-1, 2k) \mid 1 \leq k \leq n/2)$
  is the first p-step of $\mathcal{R}_n^{\parallel}$ and
  $\mathcal{C}_n^{\parallel}$.
\label{lem:3.2}
\end{lemma}
\begin{proof}
  Note that $S_n^{(1)}$ is an admissible p-step, i.e., there exists a
  p-strategy having $S_n^{(1)}$ as one of its p-steps.  For example,
  the Brent and Luk strategy starts with it.

  The first pivot pair in $\mathcal{R}_n$ and $\mathcal{C}_n$ is
  $(1, 2)$, i.e., $(2k-1, 2k)$ for $k = 1$.  If all pivot pairs in
  $\mathcal{R}_n$ or $\mathcal{C}_n$ containing indices $1$ or $2$ are
  removed, the first pivot pair in the remaining sequence is $(3, 4)$,
  i.e., $(2k-1, 2k)$ for $k = 2$.  Inductively, after selecting the
  pivot pair $(2\ell-1, 2\ell)$, with $\ell < n/2$, and removing all
  pivot pairs that contain $2k-1$ or $2k$, for all
  $1 \leq k \leq \ell$, the first remaining pivot pair is
  $(2\ell'-1, 2\ell')$ for $\ell' = \ell+1$.
\qquad\end{proof}

A matrix of order $2n$ can be regarded at the same time as a block
matrix of order $n$ with $2 \times 2$ blocks (see Fig.~\ref{fig:3.1}).
As a consequence of Lemma~\ref{lem:3.2}, after the first p-step of
either $\mathcal{R}_{2n}^{\parallel}$ or $\mathcal{C}_{2n}^{\parallel}$
(i.e., $S_{2n}^{(1)}$), the diagonal $2 \times 2$ blocks are
diagonalized, and the off-diagonal blocks are yet to be annihilated.

Once we have the diagonal blocks diagonalized, it is easy to construct
the closest block p-strategy
$\widetilde{\mathcal{O}}_{2n}^{\parallel}$ from
$\mathcal{O}_n^{\parallel}$, since each pivot pair of
$\mathcal{O}_n^{\parallel}$ corresponds uniquely to an off-diagonal
$2 \times 2$ block.  A p-step of $\mathcal{O}_n^{\parallel}$ is
expanded to two successive p-steps of
$\widetilde{\mathcal{O}}_{2n}^{\parallel}$.  The expansion procedure
is given by Alg.~\ref{alg:3.2}, for
$\mathcal{O}_n \in \{\mathcal{R}_n, \mathcal{C}_n\}$, and illustrated,
for $n = 6$ and $\mathcal{O}_n = \mathcal{R}_n$, with
Fig.~\ref{fig:3.1}.  Note that a pivot pair of $\mathcal{O}_n^{\parallel}$
contributes two pairs, $(\text{\textsc{nw}}, \text{\textsc{se}})$ and
either $(\text{\textsc{ne}}, \text{\textsc{sw}})$ or
$(\text{\textsc{sw}}, \text{\textsc{ne}})$, of non-colliding
and locally closest pivot pairs in its corresponding block.

\begin{algorithm}{\small
\SetKwInOut{KwDesc}{Description}
\SetKwFunction{Append}{append}
\SetKwFunction{Even}{even}
\KwDesc{Input: $\mathcal{O}_n^{\parallel}$,
  $\mathcal{O}_n\in\{\mathcal{R}_n,\mathcal{C}_n\}$.
  Output: $\widetilde{\mathcal{O}}_{2n}^{\parallel}$.\\
  $S_n^{(i)}$ is the $i$-th p-step of $\mathcal{O}_n^{\parallel}$, and
  $S_{2n}^{(i)}$ is the $i$-th p-step of
  $\widetilde{\mathcal{O}}_{2n}^{\parallel}$.}
\BlankLine
$S_{2n}^{(1)} \leftarrow ((2k-1, 2k) \mid 1 \leq k \leq n)$\;
\For(\tcp*[f]{construct $S_{2n}^{(i)}$})
  {$i\leftarrow 2$ \KwTo $2n-1$}
  {$S_{2n}^{(i)}=\emptyset$\;
  \ForEach{$(p, q)\in S_n^{(i \bdiv 2)}$}
    {\eIf{\Even{$i$}}
      {$\text{\textsc{nw}}=(2p-1, 2q-1);\quad\text{\textsc{se}}=(2p, 2q);$\quad
       \Append $(\text{\textsc{nw}}, \text{\textsc{se}})$ to $S_{2n}^{(i)}$\;}
      {$\text{\textsc{ne}}=(2p-1, 2q);\quad\text{\textsc{sw}}=(2p, 2q-1)$\;
       \leIf{$\mathcal{O}_n=\mathcal{R}_n$}
         {\Append $(\text{\textsc{ne}}, \text{\textsc{sw}})$ to $S_{2n}^{(i)}$}
         {\Append $(\text{\textsc{sw}}, \text{\textsc{ne}})$ to $S_{2n}^{(i)}$}
    }
  }
}}%
\caption{Expansion of $\mathcal{O}_n^{\parallel}$ to
  $\widetilde{\mathcal{O}}_{2n}^{\parallel}$ for
  $\mathcal{O}_n\in\{\mathcal{R}_n,\mathcal{C}_n\}$.}
\label{alg:3.2}
\end{algorithm}

\begin{figure}[hbt]
  \begin{center}
    \includegraphics{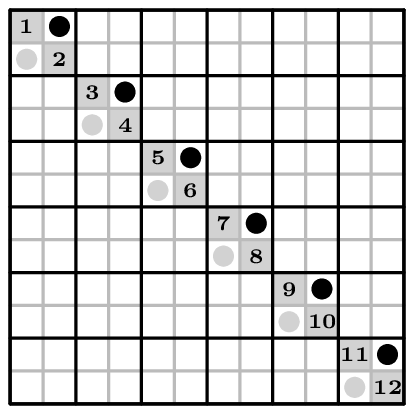}\hfill
    \includegraphics{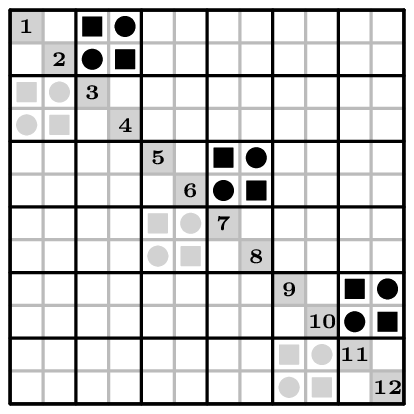}\hfill
    \includegraphics{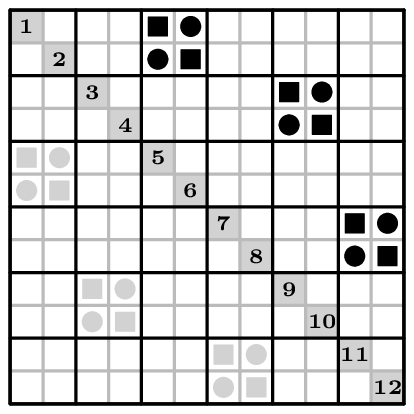}\\[9pt]
    \includegraphics{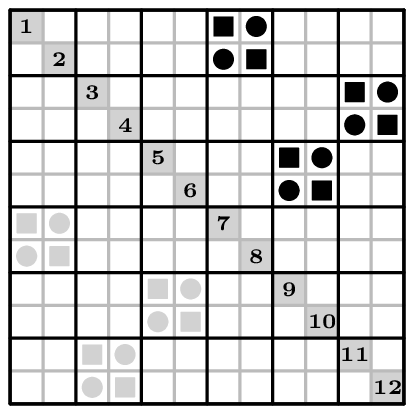}\hfill
    \includegraphics{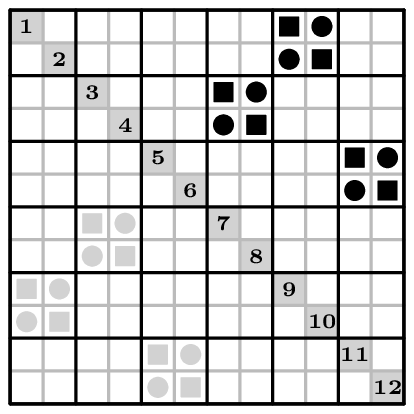}\hfill
    \includegraphics{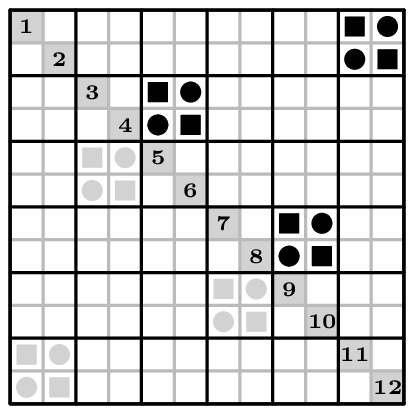}
  \end{center}
  \caption{Expansion of $\mathcal{R}_6^{\parallel}$ to
    $\mathcal{R}_{12}^{\parallel}$, according to
    Alg.~\ref{alg:3.2}.  From left to right: the black disks
    represent the odd p-steps, while the black squares stand for the
    even p-steps.}
\label{fig:3.1}
\end{figure}

It's trivial to show that, with $\mathcal{O}_n^{\parallel}$ given, the
p-strategy $\widetilde{\mathcal{O}}_{2n}^{\parallel}$ generated by
Alg.~\ref{alg:3.2} is indeed the closest \emph{block\/} p-strategy;
any other such
$\mathcal{S}_{2n}^{\parallel} \prec_{\mathcal{O}} \widetilde{\mathcal{O}}_{2n}^{\parallel}$
would induce, by the block-to-pivot correspondence, a strategy
$\mathcal{S}_n^{\parallel} \prec_{\mathcal{O}} \mathcal{O}_n^{\parallel}$,
which is impossible.  Moreover, we have verified that, for $n \leq 18$
and both $\mathcal{R}_n$ and $\mathcal{C}_n$ strategies,
$\widetilde{\mathcal{O}}_{2n}^{\parallel} = \mathcal{O}_{2n}^{\parallel}$,
and although lacking a rigorous proof, claim that the same holds for
all even $n$.  Therefore, as a tentative corrolary, to construct
$\mathcal{O}_m^{\parallel}$, for
$\mathcal{O}_m\in\{\mathcal{R}_m,\mathcal{C}_m\}$ and $m = 2^k o$,
with $k > 1$ and $o$ odd, it would suffice to construct
$\mathcal{O}_n^{\parallel}$, $n = 2o$, and apply, $k-1$ times,
Alg.~\ref{alg:3.2}.

For example, a three-level blocking algorithm for $4$ GPUs and a
matrix of order $15\cdot 1024$ requires $\mathcal{O}_8^{\parallel}$,
$\mathcal{O}_{240}^{\parallel}$, and $\mathcal{O}_{32}^{\parallel}$
strategies.  To find $\mathcal{O}_{240}^{\parallel}$, it suffices to
construct $\mathcal{O}_{30}^{\parallel}$, and expand (i.e., duplicate)
it $3$ times, since $240 = 2^3 \cdot 2 \cdot 15$.  Thus, the
$\mathcal{O}_m^{\parallel}$ strategies should be pretabulated once,
for the small, computationally feasible orders $m$, and stored into a
code library for future use.  The expansion procedure can be performed
at run-time, when the size of input is known.

The strategies just described progress from the diagonal of a matrix
outwards.  However, if the magnitudes of the off-diagonal elements in
the final sweeps of the two-sided Jacobi method are depicted, a
typical picture~\cite[page 1349]{Drmac-Veselic-2008a} shows that the
magnitudes rise towards the ridge on the diagonal.  That motivated us
to explore whether a faster decay of the off-diagonal elements far
away from the diagonal could be reached by annihilating them first,
and the near-diagonal elements last.  This change of the annihilation
order is easily done by reverting the order of pivot pairs in a sweep
of $\widetilde{\mathcal{R}}_n^{\parallel}$ and
$\widetilde{\mathcal{C}}_n^{\parallel}$.  Formally, a \emph{reverse\/}
of the strategy $\mathcal{O}_n$ is the strategy\footnote{It is also
  common to denote the reverse of $\mathcal{O}$ by
  $\mathcal{O}^{\leftarrow}$ or $\mathcal{O}_{\leftarrow}$.}
$\text{\reflectbox{$\mathcal{O}$}}_n$, given by
\begin{displaymath}
  \text{\reflectbox{$\mathcal{O}$}}_n \assgn
  ((p(\tau-k+1), q(\tau-k+1)) \mid 1 \leq k \leq \tau),
\end{displaymath}
where $\mathcal{O}_n = ((p(k), q(k)) \mid 1 \leq k \leq \tau)$.  Thus,
$\text{\reflectbox{$\mathcal{R}$}}_n^{\parallel}$ and
$\text{\reflectbox{$\mathcal{C}$}}_n^{\parallel}$ progress inwards,
ending with $S_n^{(1)}$ reversed.  We tentatively denote the reverses
of both $\mathcal{R}_n^{\parallel}$ and
$\widetilde{\mathcal{R}}_n^{\parallel}$
(resp.~$\mathcal{C}_n^{\parallel}$ and
$\widetilde{\mathcal{C}}_n^{\parallel}$) by the same symbol.

For $m=2^k$, both $\mathcal{R}_m^{\parallel}$ and
$\mathcal{C}_m^{\parallel}$ can be generated efficiently by
Alg.~\ref{alg:3.1}, since no backtracking occurs.  In this special
case it holds that
$\mathcal{R}_m^{\parallel} = \widetilde{\mathcal{R}}_m^{\parallel}$,
$\mathcal{C}_m^{\parallel} = \widetilde{\mathcal{C}}_m^{\parallel}$, and
$\mathcal{R}_m^{\parallel}$ is step-equivalent to
$\mathcal{C}_m^{\parallel}$.  The former claims are verified for
$k\leq 14$.

The respective reverses,
$\text{\reflectbox{$\mathcal{R}$}}_m^{\parallel}$ and
$\text{\reflectbox{$\mathcal{C}$}}_m^{\parallel}$, operate in the same
block-recursive fashion (preserved by Alg.~\ref{alg:3.2}) of the
Mantharam--Eberlein BR strategy~\cite{Mantharam-Eberlein-93}, i.e.,
processing first the off-diagonal block, and then simultaneously the
diagonal blocks of a matrix.  It follows that all three strategies are
step-equivalent.  Thus,
$\text{\reflectbox{$\mathcal{R}$}}_n^{\parallel}$ and
$\text{\reflectbox{$\mathcal{C}$}}_n^{\parallel}$ can be regarded as
the generalizations of the BR strategy to an arbitrary even order $n$,
albeit lacking a simple communication pattern.  Conversely, for the
power-of-two orders, $\text{\reflectbox{$\mathcal{R}$}}_m^{\parallel}$
and $\text{\reflectbox{$\mathcal{C}$}}_m^{\parallel}$ might be
replaced by the BR strategy with a hypercube-based communication.
%
%
\section{A single-GPU algorithm}\label{sec:4}
%
%
In this section we describe the two-level blocked Jacobi (H)SVD
algorithm for a single GPU\@.  The algorithm is designed and
implemented with NVIDIA CUDA~\cite{NVidiaPG} technology, but is also
applicable to OpenCL, and---at least conceptually---to the other
massively parallel accelerator platforms.

We assume that the following CUDA operations are correctly rounded, as
per IEEE 754-2008 standard~\cite{IEEE-754-2008}: ${+}$, ${-}$, ${*}$,
${/}$, $\sqrt{x}$, $\FMA(x, y, z) = x \cdot y + z$, and $\RCP(x)=1/x$.
Under that assumption, the algorithm may work in any floating-point
precision available, but is tested in double precision only, on Fermi
and Kepler GPU architectures.

The algorithm performs all computation on a GPU, and consists of
$3$ kernels:
\begin{compactenum}
\item \texttt{initV} -- optional initialization of the matrix $V$ to
  $I_n$, if the full (H)SVD is requested (for the HSVD, $V^{-T} = J V
  J$ will be accumulated instead of $V$);
\item \texttt{pStep} -- invoked once for each p-step in a block-sweep;
\item \texttt{Sigma} -- a final singular value extraction
  ($\sigma_i = \| g'_i \|_2$).
\end{compactenum}
The CPU is responsible only for the main control flow, i.e., kernel
invocations and testing of the stopping criterion.  Besides a simple
statistics from each \texttt{pStep} call, there is no other CPU
$\leftrightarrow$ GPU data transfer.  We assume that the input factor
$G$ is preloaded onto a GPU\@.
We keep the matrix $J$ partitioned as $J = \diag(I, -I)$ and represent
it with a parameter $n_+$, where $n_+$ denotes the number of positive
signs in $J$.  The output data remaining on the GPU are
$G' = U \Sigma$ (overwrites $G$), $\Sigma$, and (optionally) $V$.

Data layout (i.e., array order) is column-major (as in Fortran), to be
compatible with (cu)BLAS and other numerical libraries, like MAGMA\@.
We write one-based array indices in parentheses, and zero-based ones
in square brackets.

The matrices $G$ and $V$ are partitioned into $\mathsf{b} = n / 16$
block-columns, as in~(\ref{2.6}).  To simplify the algorithm, $n$ must
be divisible by $16$ and $\mathsf{b}$ must be even.  Otherwise, if $n$
is not divisible by $32$, $G$ has to be bordered as
in~\cite[eq.~(4.1)]{Novakovic-SingerSanja-2011}.  The reason lies in
the hardware constraints on the GPU shared memory configurations and
a fixed warp size ($32$ threads), as explained in the sequel.

We focus on \texttt{pStep} kernel, since the other two are
straightforward.  An execution grid for \texttt{pStep} comprises
$\mathsf{b} / 2$ thread blocks, i.e., one thread block per a pivot
block-pair.  Each 2-dimensional thread block is assigned
$32 \times 16 = 512$ threads, and $16\rm\ kB$ of the GPU shared
memory.  That imposes a theoretical limit of $3$ thread blocks per
multiprocessor, $100\%$ occupancy on a Fermi, and $75\%$ occupancy on
a Kepler GPU.

A chosen block p-strategy $\mathcal{S}_{\mathsf{b}}$ is preloaded in the
constant or global GPU memory in a form of an $O(1)$ lookup table
$S_{\mathsf{b}}$.  A thread block $\mathsf{t}$ in a \texttt{pStep}
invocation $\mathsf{s}$ during a block-sweep $\mathsf{r}$ obtains from
$S_{\mathsf{b}}$ the indices $\mathsf{p}$ and $\mathsf{q}$,
$1 \leq \mathsf{p} < \mathsf{q} \leq \mathsf{b}$, of the block-columns
$\mathsf{t}$ is about to process.  In other words, a mapping
$(\mathsf{r}, \mathsf{s}, \mathsf{t}) \mapsto (\mathsf{p}, \mathsf{q}) \in \mathcal{S}_{\mathsf{b}}$
establishes a correspondence between the thread blocks and the pivot
block-pairs.

A thread block behavior is uniquely determined by the block indices
$\mathsf{p}$ and $\mathsf{q}$, since the thread blocks in every
\texttt{pStep} invocation are mutually independent.  Computation in a
thread block proceeds in the three major phases:
\begin{compactenum}
\item \texttt{factorize} -- shortens the pivot block-pair
  $\begin{bmatrix}G_{\mathsf{p}} & G_{\mathsf{q}}\end{bmatrix}$,
  according to (\ref{2.7}) or (\ref{2.8}), into the triangular factor
  $R_{\mathsf{p}\mathsf{q}}$ of order $32$, and initializes
  $V_{\mathsf{p}\mathsf{q}}' = I_{32}$;
\item \texttt{orthogonalize} -- orthogonalizes
  $R_{\mathsf{p}\mathsf{q}}$ by the inner pointwise Jacobi method,
  according to the block-oriented or the full block variant of the
  Jacobi (H)SVD algorithm (see Section \ref{sec:2}),
  accumulating the applied rotations into $V_{\mathsf{p}\mathsf{q}}'$;
\item \texttt{postmultiply} -- postmultiplies
  $\begin{bmatrix}G_{\mathsf{p}} & G_{\mathsf{q}}\end{bmatrix}$, and
  optionally
  $\begin{bmatrix}V_{\mathsf{p}} & V_{\mathsf{q}}\end{bmatrix}$, by
  $V_{\mathsf{p}\mathsf{q}}'$, according to (\ref{2.9}).
\end{compactenum}
The matrices $R_{\mathsf{p}\mathsf{q}}$ and
$V_{\mathsf{p}\mathsf{q}}'$ reside in the shared memory, and together
occupy $16\rm\ kB$.  The entire allocated shared memory may also be
regarded as a single $64 \times 32$ double precision matrix, named
$G_{\mathsf{p}\mathsf{q}}$, of which $R_{\mathsf{p}\mathsf{q}}$
aliases the lower, and $V_{\mathsf{p}\mathsf{q}}'$ the upper half.

There is no shared memory configuration that can hold two square
double precision matrices of order that is a larger multiple of the
warp size than $32$.  It is therefore optimal to use the smallest
shared memory configuration ($16\rm\ kB$), leaving the highest amount
($48\rm\ kB$) of the L1 cache available.  Also, since
$R_{\mathsf{p}\mathsf{q}}$ and $V_{\mathsf{p}\mathsf{q}}'$ have to be
preserved between the phases, all phases need to be completed in the
same kernel invocation.  That creates a heavy register pressure, which
is the sole reason why only one thread block (instead of $3$) can be
active on a multiprocessor.

In the complex case ($\mathbb{F} = \mathbb{C}$), the shared memory
configuration would be $48\rm\ kB$ (suboptimal, $16\rm\ kB$
unutilized) or $32\rm\ kB$, for a Fermi or a Kepler GPU,
respectively.

We present two approaches for \texttt{factorize}.  The Cholesky
factorization of the Gram matrix $H_{\mathsf{p}\mathsf{q}}$, as in
(\ref{2.7}), is described in Subsection~\ref{sec:4.1} in two
subphases, and the QR factorization (\ref{2.8}) is discussed in
Subsection~\ref{sec:4.2}.
%
%
\subsection{The Cholesky factorization}\label{sec:4.1}
%
%
The first subphase of \texttt{factorize} loads the successive
$64 \times 32$ chunks of
$\begin{bmatrix}G_{\mathsf{p}} & G_{\mathsf{q}}\end{bmatrix}$ into
$G_{\mathsf{p}\mathsf{q}}$.  For a thread with the Cartesian indices
$[x,y]$ in our thread block, $x$ is its lane ID, and $y$ its warp
ID\@.  Let $y' = y + 16$ onwards.  After a chunk is loaded, the thread
$[x,y]$ then updates $H_{\mathsf{p}\mathsf{q}}[x,y]$ and
$H_{\mathsf{p}\mathsf{q}}[x,y']$ (kept in its registers and being
initially $0$),
\begin{displaymath}
H_{\mathsf{p}\mathsf{q}}[x,y] \mathop{{+}{=}} G_{\mathsf{p}\mathsf{q}}[{:},x]^T G_{\mathsf{p}\mathsf{q}}[{:},y],\quad
H_{\mathsf{p}\mathsf{q}}[x,y'] \mathop{{+}{=}} G_{\mathsf{p}\mathsf{q}}[{:},x]^T G_{\mathsf{p}\mathsf{q}}[{:},y'].
\end{displaymath}

Finally, when all the chunks are processed, $H_{\mathsf{p}\mathsf{q}}$
is written into $R_{\mathsf{p}\mathsf{q}}$, and
$V_{\mathsf{p}\mathsf{q}}'$ is set to $I_{32}$.  Note that data in the
GPU RAM are accessed only once.  No symmetrization is needed for
$H_{\mathsf{p}\mathsf{q}}$, since only its lower triangle is taken as
an input for the Cholesky factorization.  For details of this subphase
see Alg.~\ref{alg:4.1}.

\begin{algorithm}{\small
\SetKwInOut{KwDesc}{Description}
\SetKwFunction{Sync}{\_\_syncthreads}
\SetKwFor{For}{for}{}{endfor}
\SetKwFor{uFor}{unrolled for}{}{endfor}
\KwDesc{Input: $\begin{bmatrix}G_{\mathsf{p}} & G_{\mathsf{q}}\end{bmatrix}$.
  Output: $H_{\mathsf{p}\mathsf{q}}$.  Thread ID: $[x,y]$.}
\BlankLine
$x' = x + 32;\quad y' = y + 16;\quad h_{xy} = h_{xy'} = 0$%
\tcp*{$H_{\mathsf{p}\mathsf{q}}$ elements kept in registers}
\For(\tcp*[f]{process the next chunk of $\begin{bmatrix}G_{\mathsf{p}} & G_{\mathsf{q}}\end{bmatrix}$})%
{$(i = x;\, i < n;\, i \mathop{{+}{=}} 64)$}%
{$G_{\mathsf{p}\mathsf{q}}[x,y] = G_{\mathsf{p}}[i,y];\quad
G_{\mathsf{p}\mathsf{q}}[x,y'] = G_{\mathsf{q}}[i,y]$%
\tcp*{load the first $32$ chunk rows}
\eIf(\tcp*[f]{load the remaining chunk rows})%
{$(i' = i + 32) < n$}%
{$G_{\mathsf{p}\mathsf{q}}[x',y] = G_{\mathsf{p}}[i',y];\quad
G_{\mathsf{p}\mathsf{q}}[x',y'] = G_{\mathsf{q}}[i',y]$\;}%
(\tcp*[f]{border $G_{\mathsf{p}\mathsf{q}}$ with zeros})%
{$G_{\mathsf{p}\mathsf{q}}[x',y] = 0;\quad
G_{\mathsf{p}\mathsf{q}}[x',y'] = 0$\;}
\Sync{}\tcp*{ensure the shared memory writes have taken effect}
\uFor(\tcp*[f]{compute the partial dot-products})%
{$(j = 0;\, j < 64;\, {+}{+}j)$}%
{$j' = (x + j) \bmod 64$\tcp*{modular row addressing avoids bank conflicts}
$g_{j'x} = G_{\mathsf{p}\mathsf{q}}[j',x];\quad
g_{j'y} = G_{\mathsf{p}\mathsf{q}}[j',y];\quad
g_{j'y'} = G_{\mathsf{p}\mathsf{q}}[j',y']$\;
$h_{xy} = \FMA(g_{j'x}, g_{j'y}, h_{xy})$\tcp*{update $H_{\mathsf{p}\mathsf{q}}[x,y]$}
$h_{xy'} = \FMA(g_{j'x}, g_{j'y'}, h_{xy'})$\tcp*{update $H_{\mathsf{p}\mathsf{q}}[x,y']$}
}%
\Sync{}\tcp*{ensure that $G_{\mathsf{p}\mathsf{q}}$ is free to be overwritten}
}%
$R_{\mathsf{p}\mathsf{q}}[x,y] = h_{xy};\quad
R_{\mathsf{p}\mathsf{q}}[x,y'] = h_{xy'}$%
\tcp*{store (unsymmetrized) $H_{\mathsf{p}\mathsf{q}}$ to $R_{\mathsf{p}\mathsf{q}}$}
\Sync{}\tcp*{ensure the shared memory writes have taken effect}
}%
\caption{Device function that computes the Gram matrix
  $H_{\mathsf{p}\mathsf{q}}$.}
\label{alg:4.1}
\end{algorithm}

On a Kepler GPU, with $8$ bytes wide shared memory banks, each thread
in a warp can access a different bank.  Due to column-major data
layout, each of $32$ consecutive (modulo $64$) rows of
$G_{\mathsf{p}\mathsf{q}}$ belongs to a separate bank.  Therefore, the
modular row addressing of Alg.~\ref{alg:4.1} guarantees
bank-conflict-free operation on a Kepler GPU, and generates $2$-way
bank conflicts on a Fermi GPU, with $4$ bytes wide banks.

The next subphase consists of the in-place, forward-looking Cholesky
factorization of $H_{\mathsf{p}\mathsf{q}}$ without pivoting, i.e.,
$H_{\mathsf{p}\mathsf{q}} = L_{\mathsf{p}\mathsf{q}} L_{\mathsf{p}\mathsf{q}}^T$.
The factorization proceeds columnwise to avoid bank conflicts.  After
the factorization, the upper triangle of $R_{\mathsf{p}\mathsf{q}}$ is
set to $L_{\mathsf{p}\mathsf{q}}^T$, and the strict lower triangle to
zero.  This transposition is the only part of the entire algorithm
that necessarily incurs the bank conflicts.  The factorization has
$32$ steps.  The step $k$, for $0 \leq k < 32$, transforms
$H_{\mathsf{p}\mathsf{q}}[k{:},k{:}]$ in $2$ or $3$ stages:
\begin{compactenum}
\item[(a)] Compute $L_{\mathsf{p}\mathsf{q}}[k{:},k]$, overwriting
  $H_{\mathsf{p}\mathsf{q}}[k{:},k]$ (see Fig.~\ref{fig:4.1}(a)).  Only
  one warp is active.  The thread $[x,y]$ performs the following
  operations:
  \begin{compactitem}
    \item If $x = y = k$, then
      $L_{\mathsf{p}\mathsf{q}}[k,k] = \sqrt{H_{\mathsf{p}\mathsf{q}}[k,k]}$;
    \item else, if $x > y = k$, then
      $L_{\mathsf{p}\mathsf{q}}[x,k] = H_{\mathsf{p}\mathsf{q}}[x,k] / \sqrt{H_{\mathsf{p}\mathsf{q}}[k,k]}$;%
\footnote{Could possibly be faster if implemented as
$L_{\mathsf{p}\mathsf{q}}[x,k] = H_{\mathsf{p}\mathsf{q}}[x,k] * \RSQRT(H_{\mathsf{p}\mathsf{q}}[k,k])$.}
    \item else, the thread is dormant, i.e., does nothing.
  \end{compactitem}
\item[(b)] Update at most $16$ subsequent columns of
  $H_{\mathsf{p}\mathsf{q}}$.  Let $j = (k + 1) + y$.  If $x \geq j$
  and $j < 32$, then
  $H_{\mathsf{p}\mathsf{q}}[x,j] = \FMA(-L_{\mathsf{p}\mathsf{q}}[x,k], L_{\mathsf{p}\mathsf{q}}[j,k], H_{\mathsf{p}\mathsf{q}}[x,j])$,
  else do nothing (see Fig.~\ref{fig:4.1}(b)).
\item[(c)] If there are more columns remaining, let $j' = (k + 1) + y'$.
  If $x \geq j'$ and $j' < 32$, then
  $H_{\mathsf{p}\mathsf{q}}[x,j'] = \FMA(-L_{\mathsf{p}\mathsf{q}}[x,k], L_{\mathsf{p}\mathsf{q}}[j',k], H_{\mathsf{p}\mathsf{q}}[x,j'])$,
  else do nothing (see Fig.~\ref{fig:4.1}(c)).
\end{compactenum}
After each stage, a thread-block-wide synchronization
(\texttt{\_\_syncthreads}) is necessary.

\begin{figure}[hbt]
\begin{center}
\includegraphics[width=\textwidth]{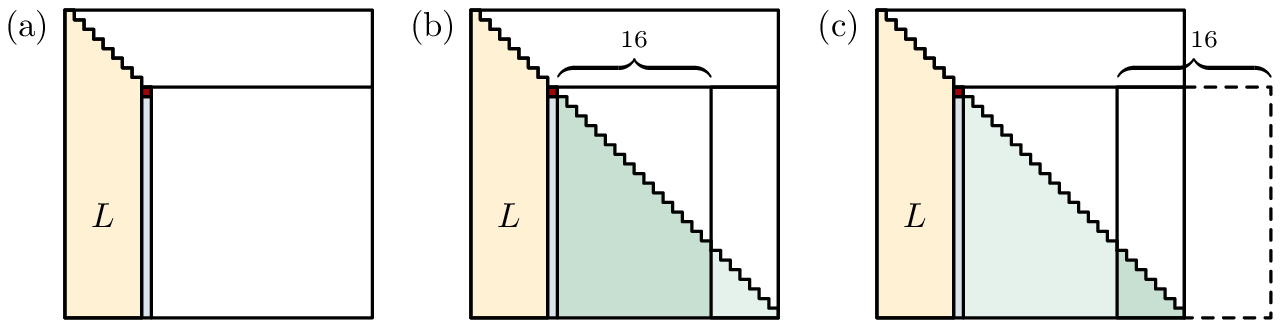}
\end{center}
\caption{The forward-looking Cholesky factorization
  $H_{\mathsf{p}\mathsf{q}} = L_{\mathsf{p}\mathsf{q}} L_{\mathsf{p}\mathsf{q}}^T$.}
\label{fig:4.1}
\end{figure}
%
%
\subsection{The QR factorization}\label{sec:4.2}
%
%
When the input matrix $G$ is badly scaled, the QR factorization
(\ref{2.8}) is required at all blocking levels, since the input
columns of too large (resp.~too small) norm could cause overflow
(resp.~underflow) while forming the Gram matrices.  If the QR
factorization employs the Householder reflectors, the column norm
computations should be carried out carefully, as detailed in
Appendix~\ref{sec:nrm}.

The tall-and-skinny in-GPU QR factorization is described
in~\cite{AndersonM-Ballard-Demmel-Keutzer-2011}.  It is applicable
when a single QR factorization per p-step is to be performed on a
GPU, e.g., in the shortening phase of a multi-GPU algorithm.  On the
shared memory blocking level, each thread block has to perform its own
QR factorization.  Therefore, an algorithm for the batched
tall-and-skinny QRs is needed in this case.

Ideally, such an algorithm should access the GPU RAM as few times as
possible, and be comparable in speed to the Cholesky factorization
approach.  We show that the algorithm can be made to access
$\begin{bmatrix}G_{\mathsf{p}} & G_{\mathsf{q}}\end{bmatrix}$ exactly
once, but the latter remains difficult to accomplish.

Let $A_0^{(i)}$ and $A_1^{(i)}$ be the $32 \times 32$ matrices that
alias the lower and the upper half of $G_{\mathsf{p}\mathsf{q}}$,
i.e., $R_{\mathsf{p}\mathsf{q}}$ and $V_{\mathsf{p}\mathsf{q}}'$,
respectively.  The first $32 \times 32$ chunk of
$\begin{bmatrix}G_{\mathsf{p}} & G_{\mathsf{q}}\end{bmatrix}$ is
loaded into $A_0^{(0)}$ and factorized as
$A_0^{(0)} = Q_0^{(0)} R_0^{(0)}$.  The factorization is performed by
$31$ successive applications of the Householder reflectors, in a
pattern similar to Fig.~\ref{fig:4.1}(b,c).  A reflector is
simultaneously computed in all active warps before the update, but is
not preserved, and $Q_0^{(0)}$ is not (explicitly or implicitly)
formed.

More precisely, for $0 \leq k < 31$, let $H_k$ be the reflector
annihilating the subdiagonal of the $k$-th column,
$H_k^{} = I_{32}^{} - \tau_k^{} w_k^{} w_k^T$, where
$w_k = \begin{bmatrix}\mathbf{0} & 1 & v_k\end{bmatrix}^T$
($\mathbf{0}$ is a vector of $k$ zeros).  In a thread with the row
index $x$, $H_k$ is represented by $\tau_k$ and $w_k[x]$.  When the
reflector is found, the warp $y$ transforms $A_0^{(0)}[\ell{:},\ell]$,
where $\ell = k + y$.  Let $z_{\ell}$ be the scalar product
$z_{\ell}^{} = w_{\ell}^T A_0^{(0)}[{:},\ell]$, computed by warp-level
shared memory reduction (on Fermi), or by warp shuffle reduction (on
Kepler).  Then, the update by $H_k$ is
\begin{displaymath}
  A_0^{(0)}[x,\ell]' = \FMA(-\tau_{\ell}^{} * z_{\ell}^{}, w_{\ell}^{}[x], A_0^{(0)}[x,\ell]).
\end{displaymath}
The transformation is then repeated for $\ell' = k + y'$.

After $R_0^{(0)}$ is formed, the second chunk of
$\begin{bmatrix}G_{\mathsf{p}} & G_{\mathsf{q}}\end{bmatrix}$ is
loaded into $A_1^{(0)}$ and similarly factored as
$A_1^{(0)} = Q_1^{(0)} R_1^{(0)}$.

The factors $R_0^{(0)}$ and $R_1^{(0)}$ are combined into $R_0^{(1)}$
by a ``peel-off'' procedure illustrated with Fig.~\ref{fig:4.2}.  The
procedure peels off one by one (super)diagonal of $R_1^{(0)}$ by the
independent Givens rotations, until (after $32$ stages) $R_1^{(0)}$ is
reduced to a zero matrix.  In the stage $k$, the row $x$ of
$R_0^{(0)}$ and the row $x - k$ of $R_1^{(0)}$ are transformed by a
rotation determined from $R_0^{(0)}[x,x]$ and $R_1^{(0)}[x-k,x]$ to
annihilate the latter.  This is the main conceptual difference from
the tall-and-skinny QR, described, e.g.,
in~\cite{Demmel-Grigori-Hoemmen-Langou-2008,Demmel-Grigori-Hoemmen-Langou-2012},
where the combining is performed by the structure-aware Householder
reflectors.  The Givens rotations are chosen to avoid the expensive
column norm computations.

\begin{figure}[hbt]
  \begin{center}
    \includegraphics{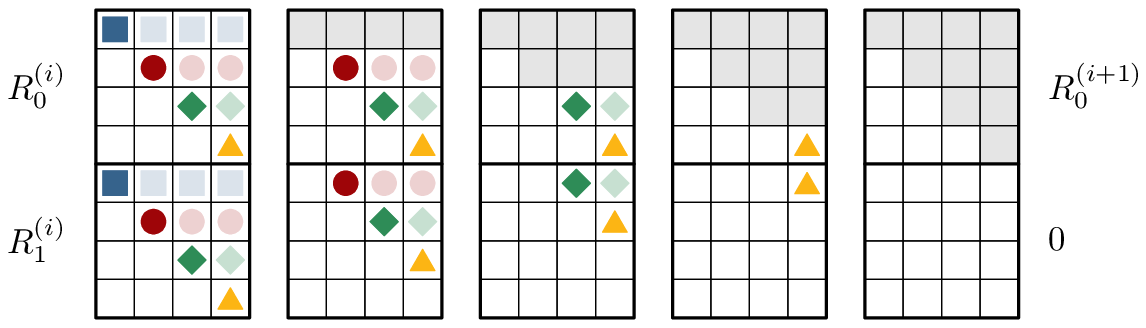}
  \end{center}
  \caption{A parallel ``peel-off'' procedure on\/ $4 \times 4$
    matrices, in\/ $4$ stages.  The rows with the same symbol are
    transformed in a stage independently of each other by the Givens
    rotations computed from the dark-colored diagonal elements.  The
    final elements of $R_0^{(i+1)}$ are fully shaded.}
  \label{fig:4.2}
\end{figure}

Each remaining chunk of
$\begin{bmatrix}G_{\mathsf{p}} & G_{\mathsf{q}}\end{bmatrix}$ is
loaded into $A_1^{(i)}$, factored as
$A_1^{(i)} = Q_1^{(i)} R_1^{(i)}$, and combined with $R_0^{(i)}$ to
obtain $R_0^{(i+1)}$.  After the final
$R_{\mathsf{p}\mathsf{q}} = R_0^{(n/32-1)}$ is formed,
$V_{\mathsf{p}\mathsf{q}}'$ is set to $I_{32}$.

Unfortunately, this approach is not recommendable when efficiency
matters.  For example, on matrices of order $3072$, the QR
factorization is $12$--$15$ times slower than the Cholesky
factorization, depending on the column norm computation algorithm.
%
%
\subsection{The orthogonalization}\label{sec:4.3}
%
%
In this phase, the inner pointwise Jacobi (H)SVD method is run on
$R_{\mathsf{p}\mathsf{q}}$.  A constant memory parameter
$M_{\mathrm{s}}$, representing the maximal number of inner sweeps,
decides whether the block-oriented ($M_{\mathrm{s}} = 1$) or the full
block variant ($M_{\mathrm{s}} > 1$, usually $M_{\mathrm{s}} = 30$)
should be performed.

The inner p-strategy $\mathcal{S}_{32}'$ is encoded into a constant
memory lookup table $S_{32}'$.  In principle, $\mathcal{S}_{32}'$ need
not be of the same type of p-strategies as $\mathcal{S}_{\mathsf{b}}$.
For example, $\mathcal{S}_{32}'$ may be of the Brent and Luk type, and
$\mathcal{S}_{\mathsf{b}}$ may be of the Matharam-Eberlein type, but
usually a choice of the p-strategies is uniform, with only a single
type for all levels.

In each p-step $s$ of an inner sweep $r$, the warp $y$ is assigned the
pivot pair $(p,q) = S_{32}'(r,s)[y]$, i.e., the pair of columns
$\begin{bmatrix}g_p & g_q\end{bmatrix}$ of $R_{\mathsf{p}\mathsf{q}}$,
and the pair of columns $\begin{bmatrix}v_p & v_q\end{bmatrix}$ of
$V_{\mathsf{p}\mathsf{q}}'$.  Then, the following subphases are
executed:
\begin{compactenum}
\item The $2 \times 2$ pivot matrix $\widehat{H}_{pq}$ from
  (\ref{2.3}) is computed.  As explained in
  Subsection~\ref{sec:4.3.1}, three dot products (for $h_{pq}$,
  $h_{pp}$, and $h_{qq}$) are needed when the rotation formulas from
  \cite{Drmac-97a} are not used.  The elements $g_p[x]$, $g_q[x]$,
  $v_p[x]$, and $v_q[x]$ are preloaded into registers of a thread with
  lane ID $x$, so, e.g., $V_{\mathsf{p}\mathsf{q}}'$ may be
  overwritten as a scratch space for warp-level reductions on Fermi
  GPUs.
\item The relative orthogonality criterion for $g_p$ and $g_q$ is
  fulfilled if
  \begin{displaymath}
    |h_{pq}| < c(\varepsilon) \sqrt{h_{pp}} \sqrt{h_{qq}} = c(\varepsilon) \| g_p \|_2 \| g_q \|_2,
  \end{displaymath}
  where $c(\varepsilon) = \varepsilon \sqrt{\hat{n}}$, and $\hat{n}$ is
  the matrix order (here, $\hat{n} = 32$).  If $g_p$ and $g_q$ are
  relatively orthogonal, then set an indicator $\rho_s$, that
  determines whether a rotation should be applied, to $\rho_s = 0$,
  else set $\rho_s = 1$.  Note that $\rho_s$ is a per-thread variable,
  but has the same value across a warp.
\item Let $a_s$ be a thread-block-wide number of warps about to
  perform the rotations.  A warp has $32$ threads, so
  $a_s = (\Sigma \rho_s)/32$, where the sum ranges over all threads in
  the thread block.  Compute $a_s$ as
  $\mathop{\text{\texttt{\_\_syncthreads\_count}}}(\rho_s)/32$.
  Since the pointwise Jacobi process stops if no rotations occurred in
  a sweep, we have to increase a per-sweep counter of rotations,
  $A_r$, by $a_s$.  The counters $a_s$ and $A_r$ are kept per thread,
  but have the same value in the thread block.
\item Let the pivot indices $p$ and $q$ correspond to the columns $k$
  and $\ell$, respectively, of the input factor $G$.  If
  $k \leq n_+ < \ell$, then compute the transformation
  $\widehat{V}_{pq}$ from (\ref{2.4}) as a hyperbolic rotation
  (\ref{4.2}), else as a trigonometric one (\ref{4.1}), according to
  Subsection~\ref{sec:4.3.1}.  If $\cs\varphi \neq 1$, then set
  $\rho_s' = 1$ (a proper rotation), else leave $\rho_s' = 0$ to
  indicate that the rotation is nearly identity.  If $\rho_s$ was $0$,
  just determine if the rotation would be a trigonometric or a
  hyperbolic one, instead of computing it.
\item If the rotation is trigonometric, find the new diagonal
  elements, $h_{pp}'$ and $h_{qq}'$,
  \begin{displaymath}
    h_{pp}' = \FMA(\tan\varphi, h_{pq}, h_{pp}),\quad
    h_{qq}' = \FMA(-\tan\varphi, h_{pq}, h_{qq}).
  \end{displaymath}
  If $\rho_s = 0$ (i.e., $\widehat{V}_{pq}$ is the identity), take
  $h_{pp}' = h_{pp}$ and $h_{qq}' = h_{qq}$.
  To keep the eigenvalues sorted
  non-increasingly~\cite{Hari-SingerSanja-SingerSasa-2014},
  if $h_{pp}' < h_{qq}'$ when $\ell \leq n_+$, or
  $h_{pp}' > h_{qq}'$ when $k > n_+$, set
  $P_2 = \left[\begin{smallmatrix}0 & 1\\1 & 0\end{smallmatrix}\right]$,
  else $P_2 = I_2$.  Define
  \begin{displaymath}
    \widehat{V}_{pq}' = \widehat{V}_{pq} P_2.
  \end{displaymath}
  The eigenvalue order tends to stabilize eventually, thus no swapping
  is usually needed in the last few
  sweeps~\cite{Mascarenhas-95,Novakovic-SingerSanja-2011}.
  If the rotation is hyperbolic, to keep $J$ partitioned, set
  $\widehat{V}_{pq}' = \widehat{V}_{pq}$ (there is no sorting, and the
  new diagonal elements are not needed).  An unpartitioned $J$ could
  lead to a slower
  convergence~\cite{SingerSanja-SingerSasa-Hari-Bokulic-Davidovic-Juresic-Uscumlic-2007}.
\item Apply, per thread, $\widehat{V}_{pq}'$ to
  $\begin{bmatrix}g_p[x] & g_q[x]\end{bmatrix}$ and
  $\begin{bmatrix}v_p[x] & v_q[x]\end{bmatrix}$ from the right, and
  store the new values into shared memory.
\item Compute $b_s = (\Sigma \rho_s')/32$, similarly to subphase $3$,
  and increase a per-sweep counter of proper rotations $B_r$ by $b_s$.
  This concludes the actions of one p-step.
\end{compactenum}

After the sweep $r$ finishes, update $B$, the total number of proper
rotations in the thread block, by $B_r$.  If no more sweeps follow,
i.e., if $A_r = 0$ (no rotations have been performed in the last
sweep), or $r = M_{\mathrm{s}}$ (the maximal number of sweeps is
reached), the thread $[0,0]$ atomically adds $B$ to the global
rotation counter $\mathsf{B}$, mapped from the CPU memory.

In $\mathsf{B}$ the number of rotations from all thread blocks is
accumulated.  Note that $\mathsf{B}$ is accessible from both the CPU
and the GPU, and is the sole quantum of information needed to stop the
global Jacobi process.  More details about a GPU-wide stopping
criterion can be found in Subsection~\ref{sec:4.5}.  This ends
\texttt{orthogonalize} phase.
%
%
\subsubsection{The Jacobi rotations}\label{sec:4.3.1}
%
%
The numerically stable, state-of-the-art procedure of computing the
trigonometric Jacobi rotations is described in~\cite{Drmac-97a}.  The
procedure relies on computing the column norms reliably, as described
in Appendix~\ref{sec:nrm}.

The rest of the computation from~\cite{Drmac-97a} is straightforward
to implement.  Since the entire shared memory per thread block is
occupied, storing and updating the column scales, as
in~\cite{Anda-Park-94}, is not possible without changing the shared
memory configuration and reducing the L1 cache.  The memory traffic
that would thus be incurred overweights the two additional
multiplications by a cosine per GPU thread.  Therefore, rotations in
the following form ($\FMA$ followed by a multiplication, if
$\cos\varphi \neq 1$) are chosen,
\begin{equation}
  \begin{bmatrix}
    g_p' & g_q'
  \end{bmatrix}
  = \cos\varphi
  \begin{bmatrix}
    g_p^{} & g_q^{}
  \end{bmatrix}
  \begin{bmatrix}
    \hphantom{-}1 & \tan\varphi\\
    -\tan\varphi  & 1
  \end{bmatrix}.
\label{4.1}
\end{equation}
The hyperbolic rotations may be computed similarly to the
trigonometric ones, by adapting the ideas from~\cite{Drmac-97a}, in
the form
\begin{equation}
  \begin{bmatrix}
    g_p' & g_q'
  \end{bmatrix}
  = \cosh\varphi
  \begin{bmatrix}
    g_p^{} & g_q^{}
  \end{bmatrix}
  \begin{bmatrix}
    1 & \tanh\varphi\\
    \tanh\varphi & 1
  \end{bmatrix}.
\label{4.2}
\end{equation}

Let \texttt{DDRJAC} be a procedure that computes, as
in~\cite{Drmac-97a}, the trigonometric rotations in form (\ref{4.1})
from the column norms it obtains by calling \texttt{DRDSSQ} (see
Appendix~A).  On a Fermi GPU (matrix order $6144$), \texttt{DDRJAC} is
only $14$\% slower than a simple procedure we discuss in the sequel.
However, the protection from the input columns of too large (or too small)
norm that \texttt{DDRJAC} offers has to be complemented by the QR
factorization at all blocking levels, which is extremely expensive.

Assume instead that the Gram matrix formation, the ordinary scalar
products, and the induced norm computations never overflow.  By using
only correctly rounded arithmetic, $\cos\varphi$ and $\tan\varphi$
of~(\ref{4.1}), or $\cosh\varphi$ and $\tanh\varphi$ of~(\ref{4.2}),
may be computed as in~(\ref{4.3})--(\ref{4.6}) (an adapted version of
the Rutishauser formulas~\cite{Rutishauser-66}):
\begin{gather}
  h = h_{qq} - \mathfrak{t} \cdot h_{pp}; \quad \ct 2\varphi = \mathfrak{t} \cdot \frac{h}{2 h_{pq}};\label{4.3}\\[3pt]
  |\ct \varphi| = |\ct 2\varphi| + \sqrt{\FMA(\ct 2\varphi, \ct 2\varphi, \mathfrak{t})};\label{4.4}\\[3pt]
  \tn\varphi = \sgn(\ct 2\varphi) \cdot \RCP(|\ct \varphi|);\label{4.5}\\
  \cs_1 \varphi = \RCP(\sqrt{\FMA(\mathfrak{t} \cdot \tn\varphi, \tn\varphi, 1)});\quad
  \cs_2 \varphi = \frac{|\ct \varphi|}{\sqrt{\FMA(|\ct \varphi|, |\ct \varphi|, \mathfrak{t})}}.\label{4.6}
\end{gather}

Formulas (\ref{4.3})--(\ref{4.6}), for $\mathfrak{t} = 1$, produce
the parameters of a trigonometric rotation ($\ct=\cot$, $\tn=\tan$, $\cs=\cos$),
and for $\mathfrak{t} = -1$, of a hyperbolic rotation ($\ct=\coth$, $\tn=\tanh$, $\cs=\cosh$).
If, numerically, $|\coth 2\varphi| = 1$, it is substituted by $5 / 4$
(see~\cite{Veselic-93}).

If $|\cot 2\varphi| < \sqrt{\varepsilon}$, then
$\sqrt{\FMA(\cot 2\varphi, \cot 2\varphi, 1)} = 1$, and~(\ref{4.4}) in
the trigonometric case simplifies to
$|\cot\varphi| = |\cot 2\varphi| + 1$.  If
$|\ct 2\varphi| \geq \sqrt{2 / \varepsilon}$, then (barring an
overflow)
$\sqrt{\FMA(\ct 2\varphi, \ct 2\varphi, \mathfrak{t})} = |\ct 2\varphi|$,
with~(\ref{4.4}) and~(\ref{4.6}) simplifying to
$|\ct\varphi| = 2 \cdot |\ct 2\varphi|$ and $\cs\varphi = 1$,
respectively.  These simplifications avoid taking square roots and a
possible overflow of $\ct^2 2\varphi$, at a price of at most $3$
floating-point comparisons.  If $|\ct 2\varphi| \leq \nu/4$, then
$\tn$ is normalized.

In~(\ref{4.6}) there are two mathematically (but not numerically)
equivalent expressions, $\cs_1 \varphi$ and $\cs_2 \varphi$, which
compute the cosine.  By a similar analysis as above,
$|\ct\varphi| \geq \sqrt{2 / \varepsilon}$ implies
$\cs_2 \varphi = 1$, $\tn \varphi \leq \sqrt{\varepsilon / 2}$, and
therefore $\cs_1 \varphi = 1$.  Testing that condition also avoids an
overflow of $\ct^2 \varphi$.  Motivated by the preliminary results
described in Appendix~\ref{sec:rot}, we have chosen the $\cs_2\varphi$
formula for our implementation.
%
%
\subsection{The postmultiplication}\label{sec:4.4}
%
%
This phase postmultiplies
$\begin{bmatrix}G_{\mathsf{p}} & G_{\mathsf{q}}\end{bmatrix}$
and, optionally,
$\begin{bmatrix}V_{\mathsf{p}} & V_{\mathsf{q}}\end{bmatrix}$
by $V_{\mathsf{p}\mathsf{q}}'$ only if the rotation counter from
\texttt{orthogonalize} is non-zero, i.e.,
if $V_{\mathsf{p}\mathsf{q}}' \neq I_{32}$.  Let
$A_{\mathsf{p}\mathsf{q}}$ alias the first half of
$G_{\mathsf{p}\mathsf{q}}$ (i.e., $R_{\mathsf{p}\mathsf{q}}$).  The
procedure, detailed in Alg.~\ref{alg:4.2}, is based on the Cannon
parallel matrix multiplication algorithm~\cite{Cannon-69}.

\begin{algorithm}{\small
\SetKwInOut{KwDesc}{Description}
\SetKwFunction{Sync}{\_\_syncthreads}
\SetKwFor{For}{for}{}{endfor}
\SetKwFor{uFor}{unrolled for}{}{endfor}
\KwDesc{Input: $\begin{bmatrix}A_{\mathsf{p}} & A_{\mathsf{q}}\end{bmatrix}, A \in \{G, V\}$.
  Output: $\begin{bmatrix}A_{\mathsf{p}}' & A_{\mathsf{q}}'\end{bmatrix} = \begin{bmatrix}A_{\mathsf{p}} & A_{\mathsf{q}}\end{bmatrix} V_{\mathsf{p}\mathsf{q}}'$.}
\BlankLine
\For(\tcp*[f]{multiply the next chunk of $\begin{bmatrix}A_{\mathsf{p}} & A_{\mathsf{q}}\end{bmatrix}$ by $V_{\mathsf{p}\mathsf{q}}'$})%
{$(i = x;\, i < n;\, i \mathop{{+}{=}} 32)$}%
{$A_{\mathsf{p}\mathsf{q}}[x,y] = A_{\mathsf{p}}[i,y];\quad
A_{\mathsf{p}\mathsf{q}}[x,y'] = A_{\mathsf{q}}[i,y]$%
\tcp*{load the RAM chunk into $A_{\mathsf{p}\mathsf{q}}$}
\Sync{}\tcp*{ensure the shared memory writes have taken effect}
$a_{xy} = a_{xy'} = 0$\tcp*{$\begin{bmatrix}A_{\mathsf{p}}' & A_{\mathsf{q}}'\end{bmatrix}$ elements kept in registers}
$j = (y + x) \bmod 32;\quad j' = (y' + x) \bmod 32$\tcp*{initial skew modulo $32$}
\uFor(\tcp*[f]{multiply-and-cyclic-shift})%
{$(k = 0;\, k < 32;\, {+}{+}k)$}%
{$a_{xy} = \FMA(A_{\mathsf{p}\mathsf{q}}[x,j], V_{\mathsf{p}\mathsf{q}}'[j,y], a_{xy})$\tcp*{update $A_{\mathsf{p}}'[i,y]$}
$a_{xy'} = \FMA(A_{\mathsf{p}\mathsf{q}}[x,j'], V_{\mathsf{p}\mathsf{q}}'[j',y'], a_{xy'})$\tcp*{update $A_{\mathsf{q}}'[i,y]$}
$j = (j + 1) \bmod 32;\quad j' = (j' + 1) \bmod 32$\tcp*{cyclic shift modulo $32$}
}%
\Sync{}\tcp*{ensure that $A_{\mathsf{p}\mathsf{q}}$ is free to be overwritten}
$A_{\mathsf{p}}'[i,y] = a_{xy};\quad
A_{\mathsf{q}}'[i,y] = a_{xy'}$\tcp*{store the product in the RAM chunk}
}%
}%
\caption{Device function for Cannon-like postmultiplication by $V_{\mathsf{p}\mathsf{q}}'$.}
\label{alg:4.2}
\end{algorithm}

Finally, Fig.~\ref{fig:4.3} summarizes the entire \texttt{pStep}
kernel, from a perspective of the shared memory state transitions per
thread block.  The GPU RAM is accessed by one read (when no rotations
occur), or by two reads and one write per element of $G$ (and,
optionally, at most one read and write per element of $V$), with all
operations fully coalesced.  The only additional global memory traffic
are the atomic reductions necessary for the convergence criterion in
\texttt{orthogonalize}.

\begin{figure}[hbt]
  \begin{center}
    \includegraphics{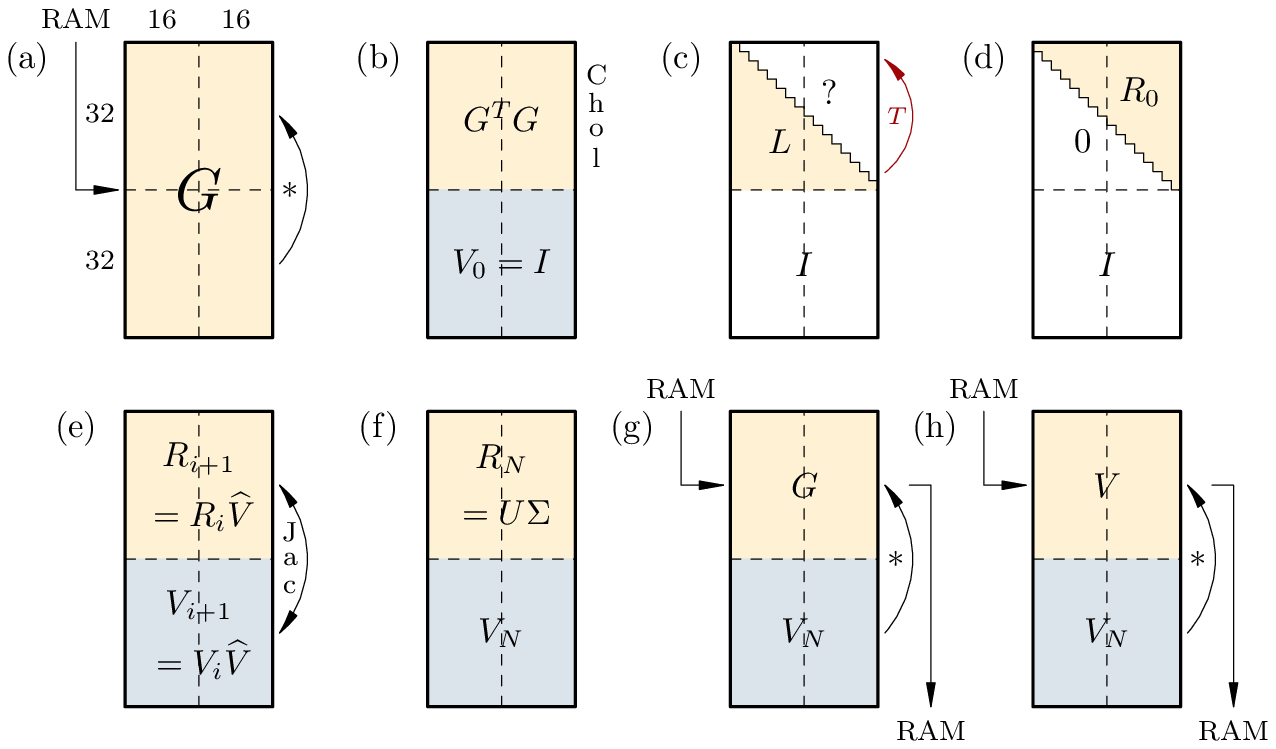}
  \end{center}
\caption{An overview of \texttt{pStep} kernel per thread block, the
  full block variant.  Each subfigure depicts the state of the shared
  memory, the computation subphase performed, and the data transfer in
  or out of the GPU RAM\@.  Subfigures (a)--(d) illustrate
  \texttt{factorize} with the Cholesky factorization, (e) and (f)
  belong to \texttt{orthogonalize}, and (g) and (h) to
  \texttt{postmultiply} phase.}
\label{fig:4.3}
\end{figure}
%
%
\subsection{A GPU-wide convergence criterion}\label{sec:4.5}
%
%
Contrary to the pointwise Jacobi algorithm, which is considered to
converge when no rotations have been performed in a sweep, stopping of
the block algorithms for large inputs is more complicated than
observing no rotations in a block sweep.  There has to be an
additional, more relaxed stopping criterion at the block level
(motivated in the sequel), while keeping the one at the inner
($32 \times 32$) level unchanged.

The columns addressed by a pivot block-pair should be relatively
orthogonal after completion of the \texttt{pStep} call in the
full-block variant, but instead they may have departed from
orthogonality, because
(see~\cite{Hari-SingerSanja-SingerSasa-2014,SingerSanja-SingerSasa-Novakovic-Uscumlic-Dunjko-2012})
\begin{compactenum}
\item the accumulated rotations are not perfectly ($J$-)orthogonal,
  and
\item the postmultiplication introduces rounding errors.
\end{compactenum}
Independently from that, in all variants, even the numerically
orthogonal columns, when subjected to \texttt{factorize} (and its
rounding errors), may result in the shortened ones that fail the
relative orthogonality criterion.

If an orthogonality failure results from the first two causes, the
ensuing rotations might be justified.  However, if the failure is
caused only by the rounding errors of the factorization process, the
spurious rotations needlessly spoil the overall convergence.

To overcome this problem, we devised a simple heuristics to avoid
excessive block sweeps with just a few rotations.  We expect these
rotations not to be proper, i.e., to have very small angles.  Let
$\mathsf{B}$ be a counter in the CPU RAM, mapped to the GPU RAM\@.
The counter is reset at the beginning of each block sweep, and is
updated from the \texttt{pStep} calls as described in
Subsection~\ref{sec:4.3}.  At the end of a block sweep, $\mathsf{B}$
contains the total number of proper rotations in that sweep.  If
$\mathsf{B} = 0$, or the maximal number of block sweeps has been
reached without convergence, the process stops.

This heuristics may skip over a relatively small number of legitimate
rotations, but nevertheless produces reasonable relative errors in the
computed singular values (see Section~\ref{sec:6}).  A reliable way of
telling (or avoiding) the exact cause of the orthogonality failures is
strongly needed in that respect.
%
%
\section{A multi-GPU algorithm}\label{sec:5}
%
%
In this Section we apply the same blocking principles one level up the
hierarchy, to the case of multiple GPUs.  As a proof-of-concept, the
algorithm is developed on a $4$-GPU Tesla S2050 system, and
implemented as a single CPU process with $4$ threads, where the thread
$0,\ldots,3$ controls the same-numbered GPU\@.  Were the GPUs
connected to multiple machines, on each machine the setup could be
similar, with a CPU process and an adequate number of threads.
Multiple processes on different machines could communicate via the
CUDA-optimized MPI subsystem.  Except replacing the inter-GPU
communication APIs, the algorithm would stay the same.

Each of $\mathsf{g}$ GPUs holds two block-columns addressed by a pivot
block-pair with a total of $\mathsf{n} = n / \mathsf{g}$ columns.  For
simplicity, we assume $n \bmod \mathsf{g} = 0$.  After an outer block
step, a single block-column on a GPU $i$ is sent to a GPU $j$, and
replaced by the one received from a GPU $k$, where $j$ may or may not
be equal to $k$, according to a block-column mapping
$S_{2\mathsf{g}}^{\prime\prime}$.  For the outer p-strategy
$\mathcal{S}_{2\mathsf{g}}^{\prime\prime}$, a block-column mapping has
to be chosen such that the communication rules implied by the mapping
are (nearly) optimal for the given network topology.

For example, in our test system with $\mathsf{g} = 4$, a GPU $i$
communicates with a GPU $j$, $j = i \operatorname{xor} 1$, faster than
with the others.  We maximized the amount of fast exchanges within an
outer block sweep for
$\text{\reflectbox{$\mathcal{R}$}}_8^{\parallel}$ (equivalent to
Mantharam--Eberlein BR on a two-dimensional hypercube) to $3$, by
choosing which pivot block-pair is assigned to which GPU in each block
step.  The result is a block-column mapping shown in Fig.~\ref{fig:5.1}.

\begin{figure}[hbt]
\begin{center}
\includegraphics{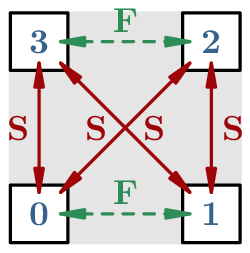}\hfill
\includegraphics{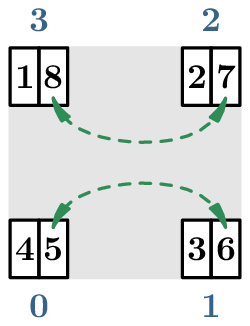}\hfill
\includegraphics{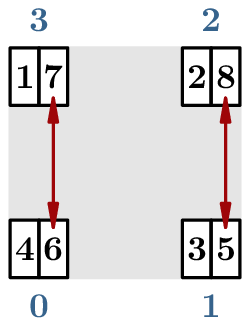}\hfill
\includegraphics{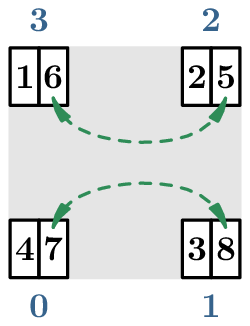}\\[6pt]
\includegraphics{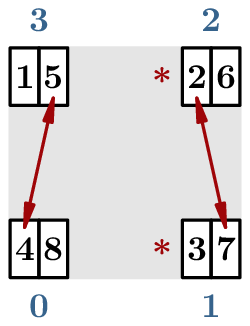}\hfill
\includegraphics{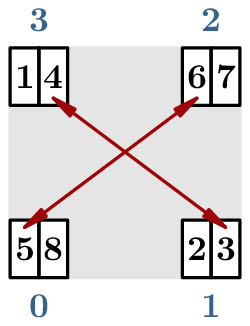}\hfill
\includegraphics{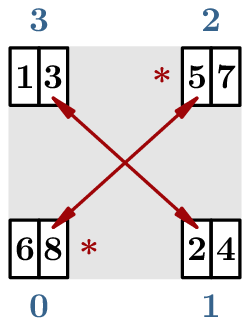}\hfill
\includegraphics{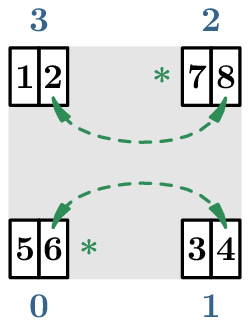}
\end{center}
\caption{The block-column mapping in a single block sweep of a
  p-strategy equivalent to Mantharam--Eberlein BR to GPUs\/
  $0,\ldots,3$.  The fast (F) communications for column exchanges
  between GPU peers are denoted by dashed curves, and the slow (S)
  exchanges by solid lines.  Two-speed communication is defined on the
  top left subfigure.  The (logical) column swaps needed to maintain
  $p < q$ are shown by an asterisk.}
\label{fig:5.1}
\end{figure}

Besides the two outer block-columns of $G$ (and, optionally, $V$),
stored in $G_A$ and $V_A$ regions of the GPU RAM, respectively, an
additional buffer space of the same size, $G_B$ and $V_B$, has to be
allocated to facilitate the BLAS 3-style matrix multiplications and
the full-duplex asynchronous communication between GPUs.  Also, for
the shortening and the single-GPU Jacobi phases, two
$\mathsf{n} \times \mathsf{n}$ matrices, $X$ and $Y$, are needed.
With a small auxiliary space AUX for the final singular value
extraction, the total memory requirements are at most
$m \times 5 \mathsf{n}$ \texttt{double} elements per GPU.

In an outer block step the following operations are performed (see
Fig.~\ref{fig:5.2}):
\begin{compactenum}
\item[(0)] form the Gram matrix $G_A^T G_A^{}$ in $X$ by
  \texttt{cublasDsyrk};
\item[(1)] factorize $G_A^T G_A^{} = R_{}^T R$ by the Cholesky
  factorization (we have chosen hybrid MAGMA's \texttt{dpotrf\_gpu},
  and this is the only place where a CPU is used for computation,
  which may be circumvented by a GPU-only implementation);
\item[(2a)] \underline{case (acc.)}: if accumulation of the product
  $\widehat{V}$ of the Jacobi rotations is desired, call a full SVD
  single-GPU Jacobi variant (the full block, the block-oriented, or a
  hybrid one) from Section~\ref{sec:4} on $X$, storing $\widehat{V}$
  in $Y$; else
\item[(2b)] \underline{case (solve)}: copy $R$ from $X$ to $Y$, call a
  partial SVD single-GPU Jacobi variant on $Y$, and solve the
  triangular linear system
  $R \widehat{V} = \widehat{U} \widehat{\Sigma}$ for
  $\widehat{V}$ by \texttt{cublasDtrsm}, with the original $R$ in $X$
  and $\widehat{V}$ overwriting $\widehat{U} \widehat{\Sigma}$ in $Y$;
\item[(3)] postmultiply $G_A$ and $V_A$ by $\widehat{V}$, using two
  \texttt{cublasDgemm} calls running in their own CUDA streams, and
  store the updated block-columns in $G_B$ and $V_B$;
\item[(4)] ensure that all GPUs have completed the local updates by a
  device-wide synchronization (\texttt{cudaDeviceSynchronize}),
  followed by a process-wide thread synchronization (wait on a common
  barrier), and a suitable MPI collective operation (e.g.,
  \texttt{MPI\_Barrier}) in the multi-process case;
\item[(5)] start, via CUDA streams, the asynchronous sends of one
  block-column from $G_B$ and the corresponding one from $V_B$ to
  another GPU, and start the asynchronous copies of the other column
  of $G_B$ and the corresponding one of $V_B$ to either the first or
  the second block-column of $G_A$ and $V_A$, according to the
  block-column mapping rules for transition to the subsequent block
  step;
\item[(6)] wait for the outstanding asynchronous operations to finish
  by the same synchronization procedure as in (4), after which a block
  step is completed.
\end{compactenum}

At the end of an outer block sweep, the threads (and processes, where
applicable) ${+}$-reduce their local counters $\mathsf{B}_i$ of proper
rotations (cf.~Subsection~\ref{sec:4.5}) to the system-wide number
$\sum_i\mathsf{B}_i$ of proper rotations performed in all block steps
in that sweep.  If the result is $0$, or the limit on the number block
sweeps has been reached, the iteration stops  and the final singular
values are extracted.

The full block variant of phases (2a) and (2b) usually has $30$ sweeps
limit for both the inner blocking and the pointwise, shared-memory
Jacobi level.  The block-oriented variant has both limits set to $1$.
Between them many hybrid variants may be interpolated.

Observe that phase (4) forces all GPUs to wait for the slowest one, in
terms of the execution of phase (2a) or (2b).  The full block variant
exhibits the largest differences in running times between GPUs,
depending on how orthogonal the block-columns are in the current block
step.  Although the full block variant is the fastest choice for a
single-GPU algorithm (see Section~\ref{sec:6}), it may be up to $35$\% 
slower in a multi-GPU algorithm than the block-oriented variant, which
has a predictable, balanced running time on all GPUs.

A reasonable hybrid variant might try to keep the running times
balanced.  A CPU thread that first completes the full block variant of
(2a) or (2b) informs immediately other threads, before proceeding to
phase (4).  The other threads then stop their inner block sweeps loops
in (2a) or (2b) when the running iteration is finished.

The wall execution times of such an approach may be even lower than
the block-oriented variant, but on the average are $10$\% higher.
Moreover, both the full block and its hybrid variant induce the larger
relative errors in $\Sigma$ than the block-oriented variant.  Such
effect may be partially explained, as in Section~\ref{sec:6}, by the
same reasons valid for the single-GPU case (more rotations applied),
but the larger differences in the multi-GPU case require further
attention.  We have therefore presented the numerical tests for the
block-oriented multi-GPU variant only.

\begin{figure}[hbt]
\begin{center}
\includegraphics{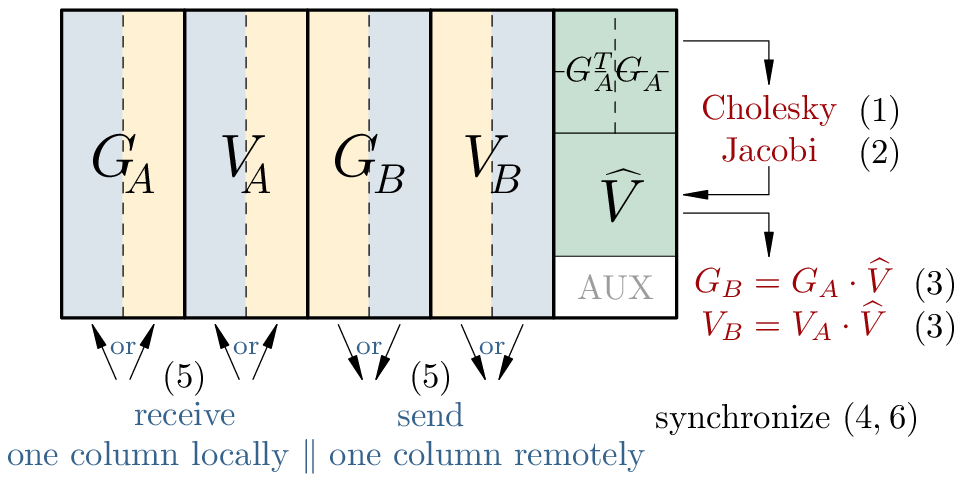}
\end{center}
\caption{Schematics of a GPU's memory organization and the outer block
  step phases.  Opera\-tions with the same number may be performed
  concurrently, in streams, if the hardware allows that.}
\label{fig:5.2}
\end{figure}
%
%
\section{Numerical testing}\label{sec:6}
%
%
In this section we define the testing data, describe the hardware, and
present the speed and accuracy results for both the single-GPU and the
multi-GPU implementations.  By these results we also confirm our
p-strategies ($\text{\reflectbox{$\mathcal{R}$}}^{\parallel}$ and
$\text{\reflectbox{$\mathcal{C}$}}^{\parallel}$)\footnote{Throughout
  this Section, we omit the subscripts indicating the matrix order on
  the p-strategies' symbols.  For each occurrence of a particular
  symbol, the matrix order is implied by the context.} from
Section~\ref{sec:3} as a reliable choice for implementing the fast
Jacobi algorithms on the various parallel architectures.

Let $\mathtt{norm}$ and $\mathtt{unif}$ be the double precision
pseudorandom number generators, such that
$\mathop{\mathtt{norm}}(\mu,\sigma)$ returns the non-zero samples from
normal distribution $\mathop{\mathcal{N}}(\mu,\sigma)$, and
$\mathop{\mathtt{unif}}(S)$ returns the samples from the continuous
uniform distribution $\mathcal{U}$ over $S$.  We have generated the
following pseudorandom spectra, for $k = 1, \ldots, 16$:
\begin{compactenum}
\item
  $\Lambda_k^{(1)}(1{:}16) = 0.5$;\quad
  $\Lambda_k^{(1)}(17{:}1024 k) = \mathop{\mathtt{norm}}(0,0.1)$,
\item
  $\Lambda_k^{(2)} = 1 + \Lambda_k^{(1)}$ (verified to be greater than zero),
\item
  $\Lambda_k^{(3)}(1{:}1024 k) = \pm\mathop{\mathtt{unif}}(\langle 10^{-7}, 10 k \rangle)$,
  where a positive or a negative sign for each $\Lambda_k^{(3)}(i)$ is
  chosen independently for $1 \leq i \leq 1024 k$ with equal
  probability,
\item
  $\Lambda_k^{(4)}(1{:}1024 k) = \mathop{\mathtt{unif}}(\langle 10^{-7}, 10 k \rangle)$.
\end{compactenum}
These arrays have been casted to the Fortran's quadruple precision
type, and denoted by $\mathbf{\Lambda}_k^{(1)}$ to
$\mathbf{\Lambda}_k^{(4)}$.  By a modified LAPACK \texttt{xLAGSY}
routine, working in quadruple precision, a set of symmetric matrices
$\mathbf{A}_k^{(j)} = \mathbf{U}_k^{(j)} \mathbf{\Lambda}_k^{(j)} [\mathbf{U}_k^{(j)}]^T$
has been generated, for $j = 1, \ldots, 4$, by pre- and post-multiplying
$\mathbf{\Lambda}_k^{(j)}$ with a product $\mathbf{U}_k^{(j)}$ of the random
Householder reflectors.  The matrices $\mathbf{A}_k^{(j)}$ have then been
factored by the symmetric indefinite factorization with the complete
pivoting~\cite{Slapnicar-98} in quadruple precision:
\begin{displaymath}
  P_k^{(j)} \mathbf{A}_k^{(j)} [P_k^{(j)}]^T =
  \widehat{\mathbf{G}}_k^{(j)} \widetilde{P}_k^{(j)} [\widetilde{P}_k^{(j)}]^T \widehat{J}_k^{(j)} \widetilde{P}_k^{(j)} [\widetilde{P}_k^{(j)}]^T [\widehat{\mathbf{G}}_k^{(j)}]^T =
  \mathbf{G}_k^{(j)} J_k^{(j)} [\mathbf{G}_k^{(j)}]^T.
\end{displaymath}
The inner permutation $\widetilde{P}_k^{(j)}$ brings
$\widehat{J}_k^{(j)}$ into $J_k^{(j)} = \diag(I, -I)$ form.  For
$j \in \{2, 4\}$ the symmetric indefinite factorization is equivalent
to the Cholesky factorization with diagonal pivoting
($J_k^{(j)} = I$).  Finally, $\mathbf{G}_k^{(j)}$ have been rounded
back to double precision, and stored as the input factors $G_k^{(j)}$,
along with $\Lambda_k^{(j)}$ and $J_k^{(j)}$.

Since one of the important applications of the (H)SVD is the eigensystem
computation of the symmetric (in)definite matrices, the procedure just
described has been designed to minimize, as much as it is computationally
feasible, the effect of the rounding errors in the factorization part.  We
have, therefore, measured the relative errors of the computed
$\Sigma^2 J$ (from $G = U \Sigma V^T$) vs.~the given
$\Lambda_k^{(j)}$ (with the elements denoted $\lambda_i$, for
$1 \leq i \leq n$), i.e.,
\begin{equation}
\max_{i=1,\ldots,n}\frac{|\fl(\sigma_i)^2 j_i - \lambda_i|}{|\lambda_i|}.
\label{6.1}
\end{equation}

It may be more natural and reliable to compute the (H)SVD of
$G_k^{(j)}$ in quadruple precision and compare the obtained singular
values with the ones produced by the double precision algorithms.
For an extremely badly conditioned $\mathbf{A}$, $\sqrt{|\Lambda|}$
may not approximate the singular values of $G$ well; e.g., if by the
same procedure as above, $\mathbf{A}$ (definite or indefinite) is
generated, with $n = 4096$ and $\kappa_2(\mathbf{A}) \geq 10^{24}$,
the resulting $G$ may have the singular values (found by the quadruple
Jacobi algorithm) differing in at least $4$--$5$ least significant
double precision digits from the prescribed singular values
$\sqrt{|\Lambda|}$.  However, for a larger $n$, the quadruple
precision SVD computation is infeasible.  We have therefore verified
accuracy of our algorithms as in (\ref{6.1}), but have generated the
modestly conditioned test matrices to avoid the problems described.

The NVIDIA graphics testing hardware, with accompanying CPUs, consists of:
\begin{compactenum}
\item[A.] Tesla C2070 (Fermi) GPU and Intel Core i7--950 CPU ($4$ cores),
\item[B.] Tesla K20c (Kepler) GPU and Intel Core i7--4820K CPU ($4$ cores),
\item[C.] Tesla S2050 (Fermi) 4 GPUs and two Intel Xeon E5620 CPUs
($2 \times 4$ cores).
\end{compactenum}
The software used is CUDA 5.5 (nvcc and cuBLAS) under $64$-bit Windows
and Linux, and MAGMA 1.3.0 (with sequential and parallel Intel MKL
11.1) under $64$-bit Linux.

As shown in Table~\ref{tbl:6.1}, the sequential Jacobi algorithm
\texttt{DGESVJ}, with the parallel MKL BLAS~1 operations, on machine C
runs approximately $16$ times slower than a single-GPU (Fermi)
algorithm for the large enough inputs.

\begin{table}[hbt]
\caption{The ratio of the wall running times of \texttt{DGESVJ} on
  machine C vs.~a single GPU (Fermi) for the Cholesky factors of the
  matrices of order $n = 1024 k$ with spectra $\Lambda_k^{(2)}$.}
\begin{center}
  {\small\begin{tabular}{@{}cc@{}}
      \toprule
      $k$ & $\texttt{DGESVJ} / \text{\reflectbox{$\mathcal{R}$}}^{\parallel}$ \\
      \midrule
      1 & \hpz 5.57 \\
      2 & \hpz 8.61 \\
      3 & 11.75 \\
      4 & 11.83 \\
      \bottomrule
  \end{tabular}
  \qquad
  \begin{tabular}{@{}cc@{}}
      \toprule
      $k$ & $\texttt{DGESVJ} / \text{\reflectbox{$\mathcal{R}$}}^{\parallel}$ \\
      \midrule
      5 & 12.34 \\
      6 & 13.47 \\
      7 & 13.62 \\
      8 & 13.58 \\
      \bottomrule
  \end{tabular}
  \qquad
  \begin{tabular}{@{}cc@{}}
      \toprule
      $k$ & $\texttt{DGESVJ} / \text{\reflectbox{$\mathcal{R}$}}^{\parallel}$ \\
      \midrule
      \hpz 9 & 14.89 \\
          10 & 15.45 \\
          11 & 15.62 \\
          12 & 16.14 \\
      \bottomrule
  \end{tabular}
  \qquad
  \begin{tabular}{@{}cc@{}}
      \toprule
      $k$ & $\texttt{DGESVJ} / \text{\reflectbox{$\mathcal{R}$}}^{\parallel}$ \\
      \midrule
      13 & 16.49 \\
      14 & 16.46 \\
      15 & 16.19 \\
      16 & 16.00 \\
      \bottomrule
  \end{tabular}}
\label{tbl:6.1}
\end{center}
\end{table}

In Table~\ref{tbl:6.2} the differences in the execution times of the
$\text{\reflectbox{$\mathcal{R}$}}^{\parallel}$ p-strategy on Fermi
and Kepler are given.  There are the three main reasons, outlined in
Section~\ref{sec:4}, why the Kepler implementation is much faster than
the Fermi one.  In order of importance:
\begin{compactenum}
\item[(i)] $8$-byte wide shared memory banks on Kepler vs.~$4$-byte
  wide on Fermi---the profiler reports $99.8$\% shared memory
  efficiency on Kepler vs.~$49.8$\% on Fermi,
\item[(ii)] warp shuffle reductions on Kepler (the warp-level
  reductions do not need the shared memory workspace), and
\item[(iii)] no register spillage on Kepler, due to the larger
  register file.
\end{compactenum}
The other profiler metrics are also encouraging: the global memory
loads and stores are more than $99$\% efficient, and the warp
execution efficiency on Fermi is about $96.5$\%, which confirms that
the presented algorithms are almost perfectly parallel.

\begin{table}[hbt]
  \caption{The wall running times (in seconds) of a Fermi ($F$) vs.~a
    Kepler ($K$) GPU for the Cholesky factors of the matrices of order
    $n = 1024 k$ with spectra $\Lambda_k^{(2)}$, the full block variant.}
  \begin{center}
    {\small\begin{tabular}{@{}cccc@{}}
      \toprule
      $k$ & $\text{Kepler}\,[\text{s}]$ & $\text{Fermi}\,[\text{s}]$ & $K/F\,[\text{\%}]$ \\
      \midrule
       1 & \hpz\hpz   1.413099 & \hpz\hpz    2.376498 & 59.5 \\
       2 & \hpz\hpz   7.206334 & \hpz   12.438532 & 57.9 \\
       3 & \hpz  22.980686 & \hpz   35.783290 & 64.2 \\
       4 & \hpz  46.357804 & \hpz   84.466500 & 54.9 \\
       5 & \hpz  95.828870 &   160.382859 & 59.8 \\
       6 &  154.643361 &   261.917934 & 59.0 \\
       7 &  246.114488 &   403.150779 & 61.0 \\
       8 &  346.689433 &   621.341377 & 55.8 \\
      \bottomrule
    \end{tabular}
    \qquad
    \begin{tabular}{@{}cccc@{}}
      \toprule
      $k$ & $\text{Kepler}\,[\text{s}]$ & $\text{Fermi}\,[\text{s}]$  & $K/F\,[\text{\%}]$ \\
      \midrule
      \hpz 9 & \hpz 506.365598 & \hpz  850.279539 & 59.6 \\
      10 & \hpz 682.577101 &  1153.337956 & 59.2 \\
      11 & \hpz 904.212224 &  1545.451594 & 58.5 \\
      12 & 1148.881987 &  1970.591570 & 58.3 \\
      13 & 1439.391787 &  2500.931105 & 57.6 \\
      14 & 1809.888207 &  3158.116986 & 57.3 \\
      15 & 2196.755474 &  3820.551746 & 57.5 \\
      16 & 2625.642659 &  4662.748709 & 56.3 \\
      \bottomrule
    \end{tabular}}
  \end{center}
  \label{tbl:6.2}
\end{table}

Even though the instruction and thread block execution partial orders
may vary across the hardware architectures, the presented algorithms
are observably deterministic.  Combined with a strong IEEE
floating-point standard adherence of both the Fermi and the Kepler
GPUs, that ensures the numerical results on one architecture are
bitwise identical to the results on the other.  This numerical
reproducibility property should likewise be preserved on any future,
standards-compliant hardware.

We proceed by showing that
$\text{\reflectbox{$\mathcal{R}$}}^{\parallel}$ and
$\text{\reflectbox{$\mathcal{C}$}}^{\parallel}$ p-strategies are
superior in terms of speed to $\mathcal{R}^{\parallel}$,
$\mathcal{C}^{\parallel}$, the Brent and Luk ($\mathcal{B}$), and
modified modulus ($\mathcal{M}$) strategies, in both the definite
and the indefinite case.  By abuse of notation, we write
$\text{\reflectbox{$\mathcal{R}$}}_b^{\parallel}$ for
the block-ori\-ented variant (otherwise, we measure the full block
variant), and $\text{\reflectbox{$\mathcal{R}$}}_{4b}^{\parallel}$ for
its $4$-GPU implementation.  Fig.~\ref{fig:6.1} depicts the wall time
ratios of the other strategies
vs.~$\text{\reflectbox{$\mathcal{R}$}}^{\parallel}$.  Except
$\mathcal{B}$ and $\text{\reflectbox{$\mathcal{C}$}}^{\parallel}$, the
other strategies are consistently about $14$--$21$\% slower than
$\text{\reflectbox{$\mathcal{R}$}}^{\parallel}$, while
$\text{\reflectbox{$\mathcal{C}$}}^{\parallel}$ is almost equally fast
as $\text{\reflectbox{$\mathcal{R}$}}^{\parallel}$. Therefore, in the
sequel we have timed only
$\text{\reflectbox{$\mathcal{R}$}}^{\parallel}$.

\begin{figure}[hbt]
\begin{center}
\includegraphics{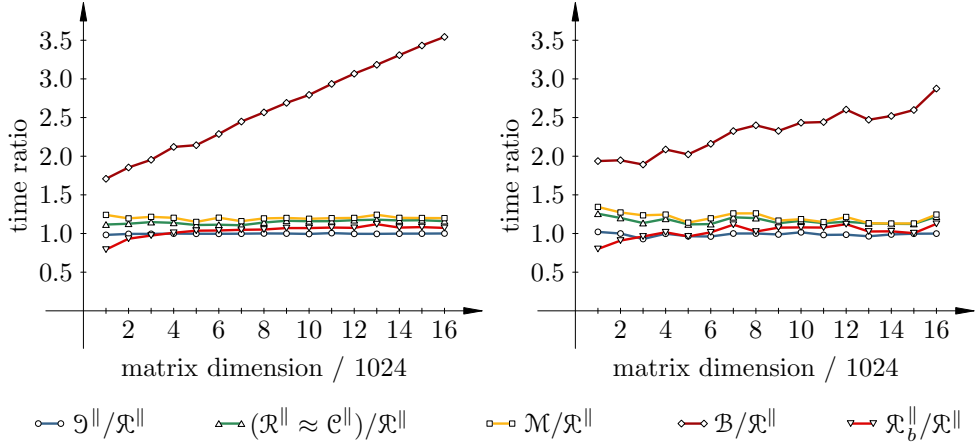}
\end{center}
\caption{The wall time ratio of the various parallel strategies on a
  single GPU (Fermi).  The test spectra for the left graph are
  $\Lambda_k^{(2)}$, and for the right graph are $\Lambda_k^{(1)}$.}
\label{fig:6.1}
\end{figure}

The standard counter-example that shows nonconvergence of the Jacobi
method under the Brent and Luk strategy for matrices of even orders
(first constructed by Hansen in~\cite{HansenE-63}, and later used
in~\cite{LukF-Park-89a}), is actually not a counter-example in the
usual diagonalization procedure which skips the rotations with the
very small angles, because there is no need for diagonalization of an
already diagonal matrix of order $2$.  On the contrary, the standard
algorithm will diagonalize this matrix in only one (second) step of
the first sweep.

However, this still does not mean that no serious issues exist
regarding convergence of the Jacobi method under $\mathcal{B}$.
Fig.~\ref{fig:6.1} indicates that a further investigation into the
causes of the extremely slow convergence (approaching $30$ block
sweeps) under $\mathcal{B}$ may be justified.

The block-oriented variant has more block sweeps and, while slightly
faster for the smaller matrices, is about $7$\% slower for the larger
matrices than the full block variant.  It may be more accurate in
certain cases (see Fig.~\ref{fig:6.2}), due to the considerably
smaller total number of the rotations performed, as shown in
Table~\ref{tbl:6.3}.  The strategies $\mathcal{M}$ and $\mathcal{B}$
are far less accurate than the new p-strategies.

\begin{figure}[hbt]
\begin{center}
\includegraphics{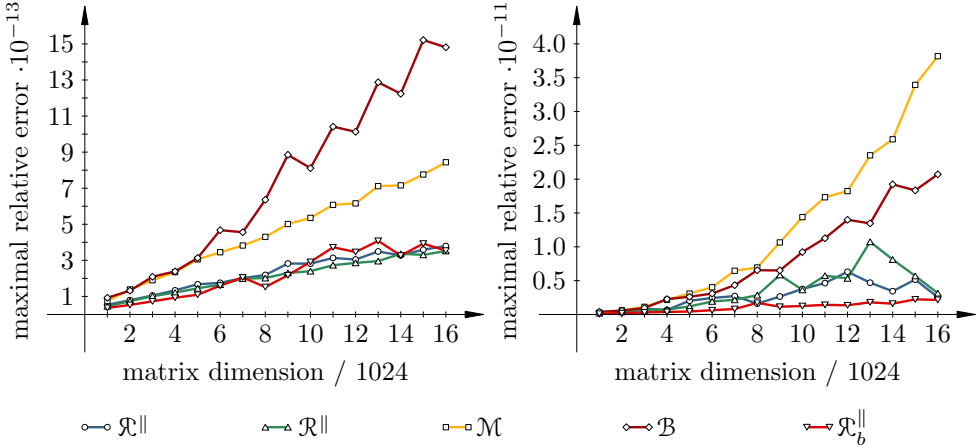}
\end{center}
\caption{The relative accuracy of the various parallel strategies on a
  single GPU.  The test spectra for the left graph are
  $\Lambda_k^{(2)}$, and for the right graph are $\Lambda_k^{(1)}$.}
\label{fig:6.2}
\end{figure}

\begin{table}[hbt]
  \caption{The number of block sweeps and the average ratios (with
    small variance) of the total number of rotations of the full block
    vs.~the block-oriented variant, per\/ $4$ spectrum types on a
    single GPU.}
  \begin{center}
    {\small\begin{tabular}{@{}lcccc@{}}
      \toprule
      Spectrum type & 1 & 2 & 3 & 4 \\
      \midrule
      average ratio of the number of rotations
      $\text{\reflectbox{$\mathcal{R}$}}^{\parallel}/\text{\reflectbox{$\mathcal{R}$}}_b^{\parallel}$
      & 2.29 & 2.10 & 2.30 & 2.08 \\
      range of the number of block sweeps $\text{\reflectbox{$\mathcal{R}$}}^{\parallel}$
      & \hpz 8--12 & 8--9\hpz & 8--11 & 7--9\hpz \\
      range of the number of block sweeps $\text{\reflectbox{$\mathcal{R}$}}_b^{\parallel}$
      & 10--14 & 9--12 & 9--14 & 9--12 \\
      \bottomrule
    \end{tabular}}
  \end{center}
  \label{tbl:6.3}
\end{table}

MAGMA's \texttt{dgesvd} routine has been tested with the sequential
(seq.)~and the parallel (par.)~($4$ threads) MKL library on machine
A\@.  The relative accuracy is identical in both cases.  Compared with
the single-GPU Fermi algorithm, MAGMA (seq.)~is $1.5$--$3$ times
slower, and MAGMA (par.)~is up to $2$ times faster.  On the other
hand, MAGMA (par.)~is, on average, $30$\%, and for the larger matrix
sizes, more than $45$\% slower than the block-oriented $4$-GPU Fermi
(solve) implementation.  The (acc.)~implementation is about $35$\%
slower than (solve) (see Fig.~\ref{fig:6.3}), and only marginally more
accurate (see Fig.~\ref{fig:6.4}).  For the matrix orders of at least
$4096$, the fastest Jacobi implementation on $4$ GPUs is about $2.7$
times faster than the fastest one on $1$ GPU.

MAGMA's accuracy is comparable to a single-GPU algorithm for the
well-con\-di\-tioned test matrices, and better than a multi-GPU
algorithm, but in the (separately tested) case of matrices with badly
scaled columns ($\kappa_2 \approx 10^{12}$), the relative errors of
MAGMA could be more than $20$ times worse than the Jacobi ones.

\begin{figure}[hbt]
\begin{center}
\includegraphics{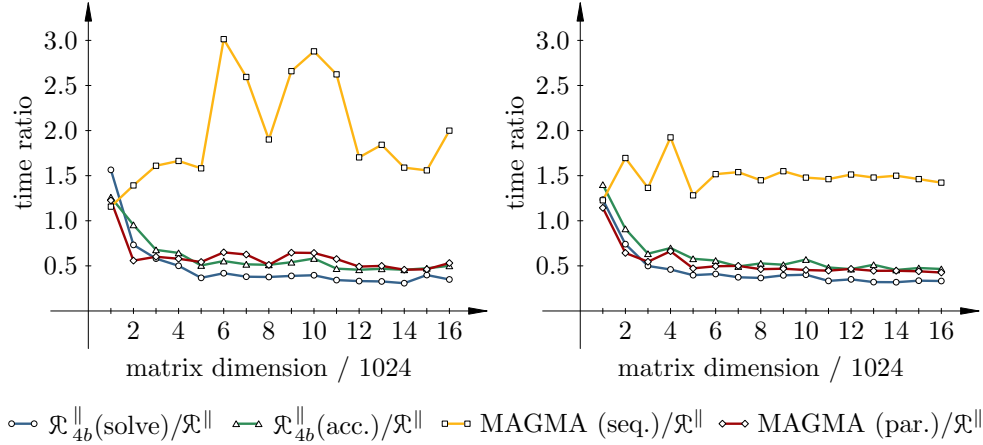}
\end{center}
\caption{The wall time ratio of the block-oriented\/ $4$-GPU Fermi
  implementations and
  M\kern-0.11ptA\kern-0.11ptG\kern-0.11ptM\kern-0.11ptA vs.~a single
  GPU.  The test spectra for the left graph are $\Lambda_k^{(2)}$, and
  for the right graph are $\Lambda_k^{(4)}$.}
\label{fig:6.3}
\end{figure}

\begin{figure}[hbt]
\begin{center}
\includegraphics{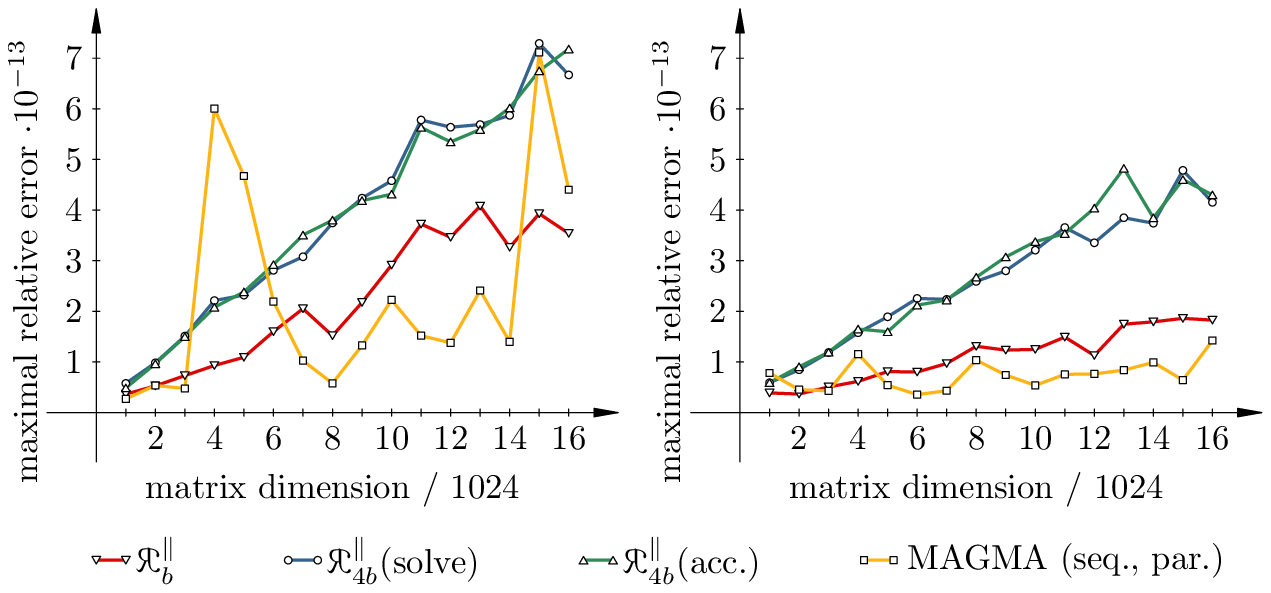}
\end{center}
\caption{The relative accuracy of the block-oriented\/ $1$- and\/
  $4$-GPU Fermi implementations and MAGMA\@.  The test spectra for the
  left graph are $\Lambda_k^{(2)}$, and for the right graph are
  $\Lambda_k^{(4)}$.}
\label{fig:6.4}
\end{figure}

Unlike MAGMA, the Jacobi GPU algorithms are perfectly scalable to an
arbitrary number of GPUs, when the matrix order is a growing function
of the number of assigned GPUs.  That makes the Jacobi-type algorithms
readily applicable on the contemporary large-scale parallel computing
machinery, which needs to leverage the potential of a substantial
amount of numerical accelerators.
%
%
\section*{Conclusions}
%
%
In this paper we have developed a set of new parallel Jacobi
strategies, both faster and more accurate than the widely used ones.
The new strategies may be seen as the generalizations of the
Mantharam--Eberlein block-recursive
strategy~\cite{Mantharam-Eberlein-93} to all even matrix orders.
These new strategies are combined with the multi-level blocking and
parallelization techniques explored
in~\cite{Hari-SingerSanja-SingerSasa-2010,Hari-SingerSanja-SingerSasa-2014,SingerSanja-SingerSasa-Novakovic-Uscumlic-Dunjko-2012,SingerSanja-SingerSasa-Novakovic-Davidovic-Bokulic-Uscumlic-2012,Novakovic-SingerSanja-2011},
to deliver the Jacobi-type (H)SVD algorithms for the graphics
processing unit(s), competitive with the leading hybrid (CPU $\!+\!$ GPU)
alternatives, like MAGMA\@.  The new algorithms are carefully designed
to use a CPU primarily as a controlling unit.  To this end, a
collection of the auxiliary shared-memory routines for the concurrent
formation of the Gram matrices, the Cholesky and QR factorizations,
and the numerically robust vector $2$-norm computations are proposed.
The numerical results confirm that in the massively parallel GPU case
the Jacobi-type methods retain all the known
advantages~\cite{Drmac-Veselic-2008,Drmac-Veselic-2008a}, while
exhibiting noteworthy speed.
%
%
{\appendix\section{Parallel norm computation}\label{sec:nrm}
%
%
An essential prerequisite for computing the Householder reflectors and
the Jacobi rotations~\cite{Drmac-97a} is obtaining the column norms
(effectively, the sums of squares) reliably, avoiding the possible
underflows and overflows of an ordinary scalar product.  However, a
strictly sequential nature of LAPACK's \texttt{DLASSQ} is unsuitable
for parallel processing.  Therefore, we propose an alternate
procedure, \texttt{DRDSSQ}, based on the parallel reduction concept.

Let $\mu$ be the smallest and $\nu$ the largest positive normalized
floating-point number, $\varepsilon$ the maximal relative roundoff
error ($\varepsilon = 2^{-53}$ for \texttt{double} with rounding to
nearest), $\gamma = 1 - \varepsilon$, $\delta = 1 + \varepsilon$, and
$x$ a vector of length $n$, with no special values ($\pm \infty$,
\texttt{NaN}s) for its components.  A floating-point approximation of
an exact quantity $\xi$ is denoted by $\rn(\xi)$, $\ru(\xi)$, or
$\rz(\xi)$, for rounding to nearest, to $+\infty$, or to $0$,
respectively.

Find $M \assgn \max_i|x_i|$.  If $M = 0$, $x$ is a zero vector.  Else,
there exists the smallest non-zero $|x_i|$, which can be computed as
$m \assgn \min_i|x_i'|$, where $x_i' = x_i$ for $|x_i| > 0$, and
$x_i' = \nu$ otherwise.  Provided enough workspace for holding, or
another technique for exchanging the partial results (such as the warp
shuffle primitives of the Kepler GPU architecture), $M$ and $m$ could
be found by a parallel min/max-reduction of $x$.

If the floating-point subnormals and infinity are supported,
inexpensive, and safe to compute with (i.e., no exceptions are raised,
or the non-stop exception handling is in effect), a sum of $x_i^2$
might also be computed.  If the sum does not overflow, and a
satisfactory accuracy is found to be maintained (e.g., the underflows
could not have happened if $\rn(m^2) \geq \mu$), \texttt{DRDSSQ} stops
here.

Otherwise, note that the depth of a reduction tree for the summation
of $x_i^2$ is $\lceil \lg n \rceil$, and at each tree level (the first
one being level $0$) at most $\delta$ relative error is accumulated.
Inductively, it follows that if, for some $s = 2^{\ell}$,
\begin{displaymath}
  2^{\lceil \lg n \rceil} (s M)^2 \delta^{(1 + \lceil\lg n\rceil)} \leq \nu,
\end{displaymath}
then the sum of $(s x_i)^2$ cannot overflow.  Also, if
$(s m)^2 \gamma \geq \mu$, for some $s = 2^k$, then no $(s x_i)^2$ can
underflow.  If some $j$ could be substituted for both $k$ and $\ell$,
it would define a scaling factor that simultaneously protects from the
potential overflows and underflows.  When the range of values of $x$
does not permit a single scaling, the independent scalings of too
large and too small values of $x$ should be performed.

Such a scaling of $x_i$ by $s$ in the binary floating-point arithmetic
introduces no rounding errors and amounts to a fast integer addition
of the exponents of $x_i$ and $s$.  Instead of the scale factors
themselves, only their exponents need to be stored and manipulated as
machine integers.  For clarity, the scales remain written herein as
the integer powers of $2$.  A pair $(s,y)$ thus represents a number
with the same precision as $y$, but with the exponent equal to a sum
of the exponents of $s$ and $y$.

As motivated above, define the safe, inclusive bounds $\tilde{\mu}$
(lower) and $\hat{\nu}$ (upper) for the values of $x$ for which no
overflow nor underflow can happen, as
\begin{displaymath}
  \tilde{\mu} = \sqrt{\mu / \gamma},\quad
  \delta_n = 2^{\lceil \lg n \rceil} \delta^{(1 + \lceil \lg n \rceil)},\quad
  \hat{\nu} = \sqrt{\nu / \delta_n}.
\end{displaymath}
Consider the following computations over a partition of the set of
values of $x$:
\begin{compactitem}
\item if $[m, M] \cap [\tilde{\mu}, \hat{\nu}] \neq \emptyset$, set
  $s_1 = 1 = 2^0$ (no scaling needed), and compute
  \begin{displaymath}
    \sigma_1^2 = \sum_{i=1}^n \bar{x}_i^2,\quad
    \bar{x}_i =
      \begin{cases}
        x_i, & \tilde{\mu} \leq |x_i| \leq \hat{\nu},\\
        0,   & \text{otherwise},
      \end{cases}
  \end{displaymath}
\item if $M > \hat{\nu}$, take the largest $s$ such that
  $s M \leq \hat{\nu}$, denote it by $s_2$, and compute
  \begin{displaymath}
    \sigma_2^2 = \sum_{i=1}^n (s_2 \hat{x}_i)^2,\quad
    \hat{x}_i =
      \begin{cases}
        x_i, & |x_i| > \hat{\nu},\\
        0, & \text{otherwise},
      \end{cases}
  \end{displaymath}
\item if $m < \tilde{\mu}$, take the smallest $s$ such that
  $s m \geq \tilde{\mu}$, denote it by $s_0$, and compute
  \begin{displaymath}
    \sigma_0^2 = \sum_{i=1}^n (s_0 \tilde{x}_i)^2,\quad
    \tilde{x}_i =
      \begin{cases}
        x_i, & |x_i| < \tilde{\mu},\\
        0, & \text{otherwise}.
      \end{cases}
  \end{displaymath}
\end{compactitem}
From $m$, $M$, $\tilde{\mu}$, $\hat{\nu}$ it is known in advance which
partial sums are necessarily $0$, and the procedure should be
simplified accordingly.  If, e.g., $m \geq \tilde{\mu}$, then
$\sigma_0^2 = 0$.

A C/C++ implementation of finding $s_0 = 2^k$ or $s_2 = 2^{\ell}$ is
remarkably simple.  An expression \texttt{y\,=\,frexp(x,\,\&e)} breaks
$\mathtt{x}$ into $0.5 \leq \mathtt{y} < 1$ and $\mathtt{e}$ such that
$2^{\mathtt{e}} \mathtt{y} = \mathtt{x}$.  Let $\mathtt{f} = m$,
$\mathtt{t} = \ru(\tilde{\mu})$, $\mathtt{j} = k$ for $s_0$, or
$\mathtt{f} = M$, $\mathtt{t} = \rz(\hat{\nu})$, $\mathtt{j} = \ell$
for $s_2$. Also, let \texttt{fy\,=\,frexp(f,\,\&fe)} and
\texttt{ty\,=\,frexp(t,\,\&te)}.  Then $\mathtt{j}$ is returned by a
code fragment:
\begin{displaymath}
\verb|j = (f <= t) ? (te - fe) + (fy < ty) : (te - fe) - (fy > ty)|.
\end{displaymath}

If there is more than one non-zero partial sum of squares, such
$(s_i^{-2}, \sigma_i^2)$ are expressed in a ``common form'',
$(\breve{s}_i^{-2}, \breve{\sigma}_i^2)$, where
$0.5 \leq \breve{\sigma}_i^2 < 2$, and the scales' exponents remain
even.  Let $(s_i^{-2}, \sigma_i^2) = (2^j, 2^m y)$, where $y$ is a
significand of $\sigma_i^2$.  Since $\sigma_i^2$ is normalized by
construction, $1 \leq y < 2$.  Define $m' = -(m \bmod 2)$ and
$j' = j + m - m'$.  Then $m' \in \{-1, 0\}$, $j'$ remains even, and
$(\breve{s}_i^{-2}, \breve{\sigma}_i^2) = (2^{j'}, 2^{m'} y)$.

The common form makes ordering the pairs by their magnitudes
equivalent to ordering them lexicographically.  First, we find the two
(out of at most three) partial sums which are the smallest by
magnitude.  We then add these partial sums together, such that the
addend smaller by magnitude is rescaled to match the scale of the
larger one.  Let
$(s_+^{-2}, \sigma_+^2) = (\breve{s}_{\leq}^{-2}, \breve{\sigma}_{\leq}^2) + (\breve{s}_>^{-2}, \breve{\sigma}_>^2)$,
with
$(\breve{s}_{\leq}^{-2}, \breve{\sigma}_{\leq}^2) \leq (\breve{s}_>^{-2}, \breve{\sigma}_>^2)$.
Then $s_+^{-2} = \breve{s}_>^{-2}$,
$s_-^{-2} = \breve{s}_{\leq}^{-2} / \breve{s}_>^{-2}$, and
$\sigma_+^2 = s_-^{-2} \breve{\sigma}_{\leq}^2 + \breve{\sigma}_>^2$.

If one more addition is needed, $(s_+^{-2}, \sigma_+^2)$ has to be
brought into the common form $(\breve{s}_+^{-2}, \breve{\sigma}_+^2)$,
and summed with the remaining addend by the above procedure.  However,
both $(s_2^{-2}, \sigma_2^2)$ and $(s_0^{-2}, \sigma_0^2)$ have to be
computed only when $n \tilde{\mu}^2 \approx \varepsilon \hat{\nu}^2$.
Such large $n$ seldom occurs.  In either case, accuracy of the final
result is maintained by accumulating the partial sums in the
nondecreasing order of their magnitudes.

The result of \texttt{DRDSSQ} is $(s^{-2}, \sigma^2)$, and the norm of
$x$ is $\|x\|_2 = \sqrt{\sigma^2} / s$.  If $\|x\|_2$ overflows or
underflows for $x$ a column of $G$, the input factor should be
initially rescaled (if possible).  A procedure similar to
\texttt{DRDSSQ} is implementable wherever the parallel reduction is a
choice (e.g., with \texttt{MPI\_Allreduce} operation).

By itself, \texttt{DRDSSQ} does not guarantee numerical
reproducibility, if the underlying parallel reductions do not possess
such guarantees.  The ideas from~\cite{Demmel-Nguyen-2013} might be
useful in that respect.
%
%
\section{A choice of the rotation formulas}\label{sec:rot}
%
%
In the block Jacobi algorithms, it is vital to preserve
($J$-)orthogonality of the accumulated $V$.  In the hyperbolic case,
the perturbation of the hyperbolic singular values also depends on the
condition number of
$V$~\cite[Proposition~4.4]{Hari-SingerSanja-SingerSasa-2014}.  A
simple attempt would be to try to compute each rotation as
($J$-)orthogonal as possible, without sacrificing performance.

Departure from a single rotation's ($J$-)orthogonality should be
checked in a sufficiently high (e.g., $128$-bit quadruple) precision,
as $d_t = |(\cos^2 \varphi + \sin^2 \varphi) - 1|$, or as
$d_h = |(\cosh\varphi - \sinh\varphi)(\cosh\varphi + \sinh\varphi) - 1|$,
with $\sin\varphi = \cos\varphi * \tan\varphi$, or
$\sinh\varphi = \cosh\varphi * \tanh\varphi$.  For each binary
exponent $-53 \leq \mathtt{e} \leq 53$ we generated, on a CPU,
$2^{24}$ uniformly distributed pseudorandom $52$-bit integers
$\mathtt{m}_i$, to form $|\ct 2\varphi|_i$ with the exponent
$\mathtt{e}$ and the non-implied bits of the significand equal to
$\texttt{m}_i$.  From $|\ct 2\varphi|_i$ and~(\ref{4.4})--(\ref{4.6})
we computed $(\tn\varphi)_i$, $(\cs_1\varphi)_i$, and
$(\cs_2\varphi)_i$ in double precision.  In the Fortran's quadruple
arithmetic the corresponding $d_t$ and $d_h$ were then found and
averaged, over all tested exponents.  The results are summarized in
Table~\ref{tbl:B.1}.

\begin{table}[hbt]
  \caption{The average departures from ($J$-)orthogonality of the
    rotations given by~(\ref{4.4})--(\ref{4.6}).}
  \begin{center}
    \begin{tabular}{@{}cccc@{}}
      \toprule
      \multicolumn{2}{c}{trigonometric rotations} & \multicolumn{2}{c}{hyperbolic rotations}\\
      $d_t$ with $\cos_1\varphi$ & $d_t$ with $\cos_2\varphi$ & $d_h$ with $\cosh_1\varphi$ & $d_h$ with $\cosh_2\varphi$\\
      \midrule
      $8.270887 \cdot 10^{-17}$ & $8.335956 \cdot 10^{-17}$ & $7.575893 \cdot 10^{-17}$ & $6.586691 \cdot 10^{-17}$\\
      \bottomrule
    \end{tabular}
  \end{center}
  \label{tbl:B.1}
\end{table}

Table~\ref{tbl:B.1} indicates that $\cosh_2\varphi$ produces, on
average, more $J$-orthogonal hyperbolic rotations than
$\cosh_1\varphi$.  In the trigonometric case it is the opposite, but
with a far smaller difference.  Orthogonality of the final $V$ was
comparable in the tests for both trigonometric versions, often
slightly better (by a fraction of the order of magnitude) using
$\cos_2\varphi$.  Therefore, $\cs_2\varphi$ formulas were chosen for a
full-scale testing.

If $\RSQRT(x) = 1 / \sqrt{x}$ were correctly rounded in CUDA,
$\cs_1\varphi$ could be written as
\begin{equation}
  \cs_1'\varphi = \RSQRT(\FMA(\mathfrak{t} \cdot \tn\varphi, \tn\varphi, 1)).
  \label{B.1}
\end{equation}
With~(\ref{B.1}) and a correctly rounded-to-nearest $\RSQRT$ prototype
CUDA implementation\footnote{Courtesy of Norbert Juffa of NVIDIA.}
there was a further improvement of orthogonality of $V$.  Although
(\ref{B.1}) has only one iterative operation ($\RSQRT$) instead of two
($\RCP$ and $\sqrt{x}$), and thus has a potential to be faster than
(\ref{4.6}), we omitted (\ref{B.1}) from the testing due to a slowdown
of about $1$\% that we expect to vanish with the subsequent
implementations of $\RSQRT$.

It is still far from conclusive which formulas from~(\ref{4.6})
or~(\ref{B.1}), and for which ranges of $\ct 2\varphi$, should be
used.  However, $\cs_2\varphi$ or $\cs_1'\varphi$ formulas might be an
alternative to the established $\cs_1\varphi$ ones.  A deeper analysis
is left for future work.}
%
%
\section*{Acknowledgments}
The author would like to thank Norbert Juffa of NVIDIA for providing
a prototype CUDA implementation of the correctly rounded-to-nearest
$\RSQRT$ function, Prof.~Hrvoje Jasak for generously giving access to
the Kepler GPUs donated by NVIDIA as a part of the Hardware Donation
Program, and Prof.~Zvonimir Bujanovi\'{c} for fruitful discussions.
Special thanks go to Prof.~Sanja Singer for drawing the figures with
MetaPost, and to Prof.~Sa\v{s}a Singer for proofreading of the
manuscript.

The author would also like to express his gratitude to the
anonymous referees for their detailed and helpful suggestions that
substantially improved the manuscript.
%
%
\bibliography{095242}
\bibliographystyle{siam}
%
%
\end{document}